\documentclass[twoside,11pt]{article}
\usepackage{graphicx}
\usepackage{amsmath,amsthm,amssymb,enumerate}
\usepackage{euscript,mathrsfs}
\usepackage[left=2cm,right=2cm,top=3.5cm,bottom=3.5cm]{geometry}
\usepackage{color}
\catcode`\@=11 \@addtoreset{equation}{section}

\catcode`\@=12

\allowdisplaybreaks

\newtheorem{Theorem}{Theorem}[section]
\newtheorem{Proposition}[Theorem]{Proposition}
\newtheorem{Lemma}[Theorem]{Lemma}
\newtheorem{Corollary}[Theorem]{Corollary}

\theoremstyle{definition}
\newtheorem{Definition}[Theorem]{Definition}

\newtheorem{Remark}[Theorem]{Remark}

\newcommand{\bTheorem}[1]{
\begin{Theorem} \label{T#1} }
\newcommand{\eT}{\end{Theorem}}

\newcommand{\bProposition}[1]{
\begin{Proposition} \label{P#1}}
\newcommand{\eP}{\end{Proposition}}

\newcommand{\bLemma}[1]{
\begin{Lemma} \label{L#1} }
\newcommand{\eL}{\end{Lemma}}

\newcommand{\bCorollary}[1]{
\begin{Corollary} \label{C#1} }
\newcommand{\eC}{\end{Corollary}}

\newcommand{\bRemark}[1]{
\begin{Remark} \label{R#1} }
\newcommand{\eR}{\end{Remark}}
\newcommand{\vuB}{\bar{{\bf u}}}
\newcommand{\bDefinition}[1]{
\begin{Definition} \label{D#1} }
\newcommand{\eD}{\end{Definition}}

\newcommand{\vrd}{\varrho_{\delta}}
\newcommand{\vtd}{\vartheta_{\delta}}
\newcommand{\vud}{\vu_{\delta}}

\newcommand{\vtB}{\bar{\vt}}

\newcommand{\bfphi}{\boldsymbol{\varphi}}

\newcommand{\bFormula}[1]{
\begin{equation} \label{#1}}
\newcommand{\eF}{\end{equation}}

\newcommand{\Ov}[1]{\bar{#1}}
\newcommand{\Oov}[1]{\overline{#1}}

\newcommand{\DC}{C^\infty_c}

\newcommand{\aleq}{\stackrel{<}{\sim}}
\newcommand{\ageq}{\stackrel{>}{\sim}}

\newcommand{\vr}{\varrho}
\newcommand{\vre}{\vr_\ep}

\newcommand{\vte}{\vt_\ep}
\newcommand{\vue}{\vu_\ep}

\newcommand{\vt}{\vartheta}
\newcommand{\vu}{\vc{u}}

\newcommand{\vQ}{\vc{q}}

\newcommand{\vc}[1]{{\bf #1}}

\newcommand{\Div}{{\rm div}}
\newcommand{\Grad}{\nabla}

\newcommand{\dx}{\,{\rm d} {x}}

\newcommand{\intO}[1]{\int_{\Omega} #1 \dx}

\newcommand{\ep}{\varepsilon}

\newcommand{\R}{\mathbb{R}}

\newcommand{\C}{\mbox{\F C}}

\def\eqldef{\overset{\text{\tiny\rm def}}{=}}
\def\tra{\mathsf{T}}
\def\C{{\mathrm{C}}}

\def\L{{\mathrm{L}}}

\def\X{{\mathrm{X}}}
\def\Y{{\mathrm{Y}}}
\def\W{{\mathrm{W}}}

\definecolor{Cgrey}{rgb}{0.85,0.85,0.85}
\definecolor{Cblue}{rgb}{0.50,0.85,0.85}
\definecolor{Cred}{rgb}{1,0,0}
\definecolor{fancy}{rgb}{0.10,0.85,0.10}

\newcommand\Cbox[2]{%
    \newbox\contentbox%
    \newbox\bkgdbox%
    \setbox\contentbox\hbox to \hsize{%
        \vtop{
            \kern\columnsep
            \hbox to \hsize{%
                \kern\columnsep%
                \advance\hsize by -2\columnsep%
                \setlength{\textwidth}{\hsize}%
                \vbox{
                    \parskip=\baselineskip
                    \parindent=0bp
                    #2
                }%
                \kern\columnsep%
            }%
            \kern\columnsep%
        }%
    }%
    \setbox\bkgdbox\vbox{
        \color{#1}
        \hrule width  \wd\contentbox %
               height \ht\contentbox %
               depth  \dp\contentbox
        \color{black}
    }%
    \wd\bkgdbox=0bp%
    \vbox{\hbox to \hsize{\box\bkgdbox\box\contentbox}}%
    \vskip\baselineskip%
}

\date{}

\begin{document}


\title{Stationary solutions of the Navier--Stokes--Fourier system
in planar domains with impermeable boundary}

\author{I.~S.~Ciuperca\thanks{Universit\'e de Lyon, Universit\'e Lyon 1, CNRS,
Institut Camille Jordan UMR 5208, 43 boulevard du 11 novembre 1918,
F--69622 Villeurbanne Cedex, France
({\tt  ciuperca@math.univ-lyon1.fr})}
\and E.~Feireisl\thanks{Institute of Mathematics of the Academy of Sciences of the Czech Republic,
\v Zitn\' a 25, CZ--115 67 Praha 1, Czech Republic
\& Institute of Mathematics, TU Berlin, Strasse des 17
Juni, Berlin, Germany ({\tt  feireisl@math.cas.cz})}
\and M.~Jai\thanks{Universit\'e de Lyon, CNRS, INSA de Lyon
Institut Camille Jordan UMR 5208, 20 Avenue A. Einstein, F--69621 Villeurbanne, France
({\tt mohammed.jai@insa-lyon.fr, apetrov@math.univ-lyon1.fr})}
\and A.~Petrov$^{\ddagger}$}

\pagestyle{myheadings}
\thispagestyle{plain}
\markboth{I.~S.~Ciuperca, E.~Feireisl, M.~Jai, A.~Petrov}
{\small{The Navier--Stokes--Fourier system with impermeable boundary}}

\maketitle

\begin{abstract}

The existence of weak solutions to the Navier--Stokes--Fourier system describing the
stationary states of a compressible, viscous, and heat conducting fluid in bounded
$2\mathrm{D}$-domains is shown under fairly general and physically relevant 
constitutive relations.
The equation of state of a real fluid is considered, where the admissible range of
density is confined to a bounded interval (hard sphere model). The transport
coefficients depend on the temperature in a general way including both gases and
liquids behavior. The heart of the paper are new {\it a priori} bounds resulting from
Trudinger--Moser inequality.

\end{abstract}


\tableofcontents

\section{Introduction}
\label{I}

Let $\Omega \subset \R^2$ be a bounded domain with smooth boundary.
Stationary states of a compressible, viscous, and heat conducting fluid contained
in $\Omega$ are described through the phase variables - the density $\vr \eqldef\vr(x)$,
the velocity field $\vu \eqldef \vu(x)$,
and the (absolute) temperature $\vt \eqldef \vt(x)$ - satisfying
the \emph{Navier--Stokes--Fourier system}:
\begin{subequations}
\label{I1}
\begin{align}
\label{I1a}
&\Div (\vr \vu) = 0,\\
\label{I1b}
&\Div (\vr \vu \otimes \vu) + \Grad p(\vr, \vt) = \Div (\mathbb{S}(\vt, \Grad \vu)) + \vr \vc{f},\\
\label{I1c}
&\Div \Bigl( \Bigr[ \frac{1}{2} \vr |\vu|^2 + \vr e(\vr, \vt) \Bigl] \vu \Bigr) +
\Div (p(\vr, \vt) \vu) + \Div (\vc{q} (\vt, \Grad \vt)) = \Div(\mathbb{S}(\vt, \Grad \vu)\vu)
+ \vr \vc{f} \cdot \vu + \vr G.
\end{align}
\end{subequations}
We suppose the existence of the entropy $s\eqldef s(\vr, \vt)$ related to the internal
energy $e(\vr,\vt)$ and the pressure $p(\vr, \vt)$
through \emph{Gibbs' equation}:
\begin{equation}
\label{I2}
\vt D s = De + p D \Bigl( \frac{1}{\vr} \Bigr).
\end{equation}
The viscous stress tensor $\mathbb{S}$ is given by \emph{Newton's rheological law}:
\begin{equation}
\label{I3}
\mathbb{S}(\vt, \Grad \vu)\eqldef\mu(\vt) [ \Grad \vu + (\Grad \vu)^{\tra} - \Div \vu \mathbb{I}]
+ \lambda(\vt) \Div \vu \mathbb{I},
\end{equation}
while the heat flux $\vc{q}\eqldef \vc{q}(\vr, \Grad \vt)$ obeys \emph{Fourier's law}:
\begin{equation} \label{I4}
\vc{q} \eqldef - \kappa (\vt) \Grad \vt.
\end{equation}
Here \((\cdot)^{\tra}\) and  \( \mathbb{I}\) denote the transpose of a tensor and the identity matrix,
respectively.
The functions $\vc{f}\eqldef \vc{f}(x)$ and $G\eqldef G(x)$ represent the external
force and external heat source, respectively.
The problem is closed by the set of boundary conditions:
\begin{subequations}
\label{I5-6}
\begin{align}
\label{I5}
&\vu_{|_{\partial \Omega}} = \vuB\quad\text{and}\quad\vuB \cdot \vc{n}_{|_{\partial \Omega}} = 0,\\&
\label{I6}
\vQ \cdot \vc{n}_{|_{\partial \Omega}}= L (\vt - \vtB)_{|_{\partial \Omega}}\quad\text{with}\quad L > 0,
\end{align}
\end{subequations}
where $\vtB$ is a prescribed ``threshold'' temperature 
The interested reader may consult the monograph by Galavotti \cite{GALL} or
\cite[Chapter 1]{FeNo6a} for physical background of the
problem \eqref{I1}--\eqref{I5-6}.

Our goal is to show the existence of admissible weak solutions to the
Navier--Stokes--Fourier
system under fairly general conditions imposed on the constitutive relations.
Following \cite{CiFeJaPe}, we consider the equation of state (EOS) describing real fluids:
\begin{equation*}
p(\vr, \vt) \to \infty \quad\mbox{as}\quad \vr \to \Ov{\vr} \quad\text{with}\quad\Ov{\vr} > 0 
\ \mbox{- a positive constant.}
\end{equation*}
In particular for gases, the EOS is usually written in the form
\begin{equation*}
p(\vr, \vt) = \vr \vt (B_0 + B_1(\vr, \vt)),
\end{equation*}
where the term $B_1$ represents the deviation from the standard
\emph{Boyle--Mariotte law} active in the degenerate area. Accordingly,
we focus on EOS that can be written in the following form
\begin{equation}
\label{I9}
p(\vr, \vt) = \vr \vt h(\vr),\quad h(0) = h_0 > 0,\quad h(\vr) \to \infty \quad \mbox{as}\quad \vr \to \Ov{\vr}.
\end{equation}
Note that Kolafa EOS \cite{KLM04AEHS} as well as
Carnahan--Starling EOS \cite{CarSta69EOS} can be written as \eqref{I9}.
In accordance with Gibbs' relation \eqref{I2}, the associated internal energy $e$
is a function of the temperature only.
Here we suppose
\begin{equation}
\label{I8}
e(\vr, \vt)\eqldef c_v \vt \quad\text{with}\quad c_v > 0.
\end{equation}
More general EOS can be handled by the same approach. This issue is discussed briefly
in the concluding part.

The available literature concerning stationary states of the Navier--Stokes--Fourier
system is rather limited. Besides the results
concerning smooth solutions arising as small perturbations of known
static states, see Piasecki and Pokorn\' y \cite{PiP14SSNF},
Plotnikov, Ruban and Sokolowski \cite{PRS08ICNS, PRS09IBTE},
there is a series of papers by Novotn\' y,
Pokorn\' y, and their collaborators concerning
the existence of weak solutions for problems with large data (external forces),
\cite{NoS04MTCF}. The main
novelties achieved in the present paper compared to the above
mentioned results may be summarized as follows:
\begin{itemize}

\item General EOS of the form \eqref{I9} can be handled. In particular, the pressure vanishes
for $\vt \to 0$, which is particularly relevant to gases. There is no need to add
a "cold pressure'' component independent of $\vt$.

\item The class of transport coefficients includes
\begin{equation}
\label{I10}
\mu (\vt) \approx (1 + \vt^\alpha), \quad \lambda(\vt) \approx (1 + \vt^\alpha),\quad
\kappa (\vt) \approx (1 + \vt^\alpha),\quad
0 \leq \alpha < 1.
\end{equation}
This is relevant for both gases $\alpha = \frac{1}{2}$ and liquids $\alpha = 0$.
Moreover, they are all of the same order that corresponds
to finite Prandtl number.

\item The inhomogeneous boundary conditions for the velocity \eqref{I5} are included.

\end{itemize}

The main new ingredient of our analysis is the Trudinger--Moser inequality available
for Sobolev functions in $\W^{1,2}(\Omega)$
if $\Omega \subset \R^2$ is a planar domain. The key point is the estimate of the temperature
in the form
\begin{equation}
\label{I11}
\intO{\exp( \omega |\log(\vt)|^2)} < \infty \ \mbox{for certain}\ \omega > 0
\end{equation}
resulting from boundedness of the associated entropy production rate.
Obtaining \eqref{I11}, however, is not completely straightforward
due to the presence of external driving represented by the non--homogeneous
boundary condition \eqref{I5}. Relation \eqref{I11} is obtained via a non--standard
compactness argument based on the possibility of extending the field $\Ov{\vu}$ in $\Omega$
with
sufficiently small norm. Compactness of the density is then obtained by a combination
of the method proposed by Lions \cite{LI4}
based on the monotonicity of the pressure, and Commutator Lemma
originally introduced in \cite{EF71}
to handle the time dependent viscosity coefficients.
The reader is also refered to \cite[Lemma 3.6, p.~100]{FeNo6a}.

The paper is organized as follows. In Section \ref{M}, we collect the necessary
preliminary material, formulate principal hypotheses, and state the main result.
The existence proof follows a
multi--level approximate scheme introduced in Section \ref{A}. It consists in:
\begin{itemize}

\item introducing artificial viscosity to regularize the equation of
continuity \eqref{I1a} (small parameter
$\ep$);

\item discretizing the momentum equation \eqref{I1b}
by means of a Galerkin approximation (dimension $N$ of
the approximate space);

\item replacing the total energy balance \eqref{I1c}
by the internal energy equation (small parameter
$\delta > 0$ to augment viscosity and thermal conductivity);

\item truncating the singular pressure (truncation parameter $R$).

\end{itemize}
In Section \ref{U}, we establish uniform bounds on the family of approximate
solutions noting that their existence can be shown in a manner similar to \cite{NovPok11NSFG}.
This amounts to deriving the associated entropy balance, the validity of which can be seen as an
admissibility condition imposed on the class of weak solutions. In Section \ref{N},
we perform the limit in the Galerkin approximation. At this level, the internal energy
equation is replaced by the total energy balance, and the system is augmented by the entropy
inequality. In Section \ref{R}, we derive the pressure estimates based on the application of
the so--called Bogovskii operator and
then relax the truncation. As a result, the density here and hereafter is bounded
above by $\Ov{\vr}$. In Section \ref{E},
we perform the vanishing viscosity limit in the equation of continuity.
This a delicate but nowadays rather well understood process,
where Lions' method \cite{LI4} based on compactness of the effective
viscous flux is combined with the commutator technique introduced
in \cite{EF71}. Finally, in Section \ref{D}, we remove the regularizing terms
depending on a small parameter $\delta$. In particular,
we perform in full generality the estimates leading to the crucial bound \eqref{I11}.
The paper is concluded by a short discussion in Section \ref{C}.

\section{Main result}
\label{M}

We introduce the principal hypotheses on the data and state our main result.
Here and hereafter, we use the symbol
\begin{equation}
\begin{aligned}
&a \aleq b \ \mbox{if there exists a constant}\ c > 0 \ \mbox{such that}\ a \leq cb\\
&a \approx b \ \mbox{if}\ a \aleq b \ \mbox{and}\ b \aleq a.
\end{aligned}
\end{equation}
For $h \in \L^1(\Omega)$, we denote its integral mean by
\begin{equation*}
\{ h \}_m \eqldef \frac{1}{|\Omega|} \intO{ h }.
\end{equation*}
We also use the symbol $c_D$ to denote a generic positive constant depending only
on the data (domain, boundary
conditions, constitutive relations). When there is no confusion,
we will use simply the notation \(\X(\Omega)\)
instead of  \(\X(\Omega;\Y)\) where \(\X\) is a functional space and \(\Y\) a
vectorial space.

\subsection{Constitutive equations, external forces}

The given external fields $\vc{f}$, $G$, $\Ov{\vt}$ satisfy the following assumptions:
\begin{equation}
\label{M1}
\vc{f} \in \L^\infty(\Omega; \R^2), \quad G \in \L^\infty(\Omega),\ G \geq 0, \quad \Ov{\vt} \in
\L^\infty(\partial \Omega),\quad
{\rm ess}\inf_{\partial \Omega} \Ov{\vt} > 0.
\end{equation}
We consider the pressure in the following form: there exists a constant \(\bar{\vr}>0\) such that
\begin{subequations}
\label{M2}
\begin{align}
&p(\vr, \vt) \eqldef \vr \vt h(\vr)\ \text{where}\
h \in \C^0[0, \bar{\vr}) \cap \C^1(0, \bar{\vr}), \quad h(0) = h_0 > 0, \quad h'(\vr) \geq 0,\\
&\lim_{\vr \to \Ov{\vr}} h(\vr) = \infty.
\end{align}
\end{subequations}
In accordance with Gibbs' relation \eqref{I2}, the internal energy is taken in the form:
\begin{equation}
\label{M3}
e(\vr, \vt) \eqldef c_v \vt \ \text{with}\ c_v > 0,
\end{equation}
The transport coefficients $\mu$, $\lambda$ and $\kappa$ are continuously
differentiable functions of the temperature satisfying
\begin{equation}
\label{M4}
\mu(\vt) \approx (1 + \vt^\alpha), \quad
\lambda(\vt) \approx (1 + \vt^\alpha),\quad \kappa(\vt) \approx (1 + \vt^\alpha)\ \text{with}\
0 \leq \alpha < 1.
\end{equation}
The boundary condition \eqref{I5} are determined via a field $\Ov{\vu}$ satisfying
\begin{equation}
 \label{M4bis}
\Ov{\vu} \in \W^{1, \infty}(\Omega; \R^2), \quad \Div (\Ov{\vu}) = 0, \quad
\Ov{\vu} \cdot \vc{n}_{|_{\partial \Omega}} = 0.
\end{equation}

\subsection{Weak formulation}

Let \(M\) be such that \(0<M<|\Omega|\bar{\vr}\) and denote by \(\vr_M\eqldef \frac{M}{|\Omega|}\), 
\(0 < \vr_M<\bar{\vr}\).
Let $s\eqldef s(\vr, \vt)$ be the entropy derived from \eqref{M2}, \eqref{M3}
via Gibbs' relation \eqref{I2}, leading to
\begin{equation}
\label{M5}
s(\vr, \vt) = c_v \log(\vt) - \int_{\vr_M}^\vr \frac{h(z)}{z} {\rm d}z.
\end{equation}
The weak formulation of the Navier--Stokes--Fourier system \eqref{I1}, with the
boundary conditions \eqref{I5} and \eqref{I6}
reads:
\begin{itemize}
\item {\bf Equation of continuity}
\begin{equation}
\label{W1}
\intO{ \vr \vu \cdot \Grad \varphi } = 0 \ \mbox{for any}\ \varphi \in \W^{1,\infty}(\Omega).
\end{equation}

\item {\bf Momentum balance}
\begin{subequations}
\begin{align}
\label{W2}
&\intO{ \Bigl[ \vr \vu \otimes \vu : \Grad \bfphi + p(\vr, \vt) \Div \varphi \Bigr] } =
\intO{ \mathbb{S}(\vt, \Grad \vu) : \Grad \bfphi } - \intO{ \vr \vc{f} \cdot \bfphi }\\&\notag
\text{ for any } \bfphi \in \W^{1, \infty}_0(\Omega; \R^2),\\& \label{W3}
\vu_{|_{\partial \Omega}} = \Ov{\vu}_{|_{\partial \Omega} }
\ \mbox{in the sense of traces.}
\end{align}
\end{subequations}

\item
{\bf Total energy balance}
\begin{equation}
\label{W4}
\begin{aligned}
&\intO{ \Big( \frac{1}{2} \vr |\vu|^2 + \vr e(\vr, \vt)    \Big) \vu \cdot \Grad \varphi }
+ \intO{p (\vr, \vt) \vu \cdot \Grad\varphi   } + \intO{ \vc{q} \cdot \Grad \varphi }
\\&- \intO{ \mathbb{S} (\vt, \Grad \vu)\vu \cdot \Grad \varphi  }
= \int_{\partial \Omega} L (\vt - \Ov{\vt})
\varphi \ {\rm dS}  - \intO{ \varphi \vr G   } - \intO{\varphi \vr \vc{f} \cdot \vu}
\\&+ \intO{ \vr
 (\vu \otimes \vu) : \Grad (\varphi \Ov{\vu})}
+ \intO{ p (\vr, \vt) \Div (\varphi \Ov{\vu}) }
\\& + \intO{\vr \vc{f} \cdot
(\varphi \Ov{\vu}) } - \intO{ \mathbb{S} (\vt, \Grad \vu) : \Grad (\varphi \Ov{\vu}) }
 \end{aligned}
\end{equation}
for any $\varphi \in \W^{1,\infty}({\Omega})$.  Notice that the total energy balance is obtained
by multiplying \eqref{I1b} and \eqref{I1c} by \(\varphi\Ov{\vu}\) and \(\varphi\), respectively.

\item {\bf Entropy inequality}
\begin{equation}
\label{W5}
\begin{aligned}
\intO{ \varphi
\frac{1}{\vt} &\Bigl[ \mathbb{S} (\vt, \Grad \vu) : \Grad \vu + \frac{ \kappa (\vt) }{\vt}
|\Grad \vt |^2 \Bigr] } + \intO{ \frac{\vr}{\vt} G \varphi } \\
&\leq  \int_{\partial \Omega} \varphi L \Bigl( 1 - \frac{\Ov{\vt}}{\vt} \Bigr) {\rm dS} -
\intO{ \Bigl( \frac{\vQ (\vt, \Grad \vt) }{\vt} \Bigr) \cdot \Grad \varphi }
- \intO{( \vr  s(\vr, \vt)  \vu ) \cdot \Grad \varphi }
\end{aligned}
\end{equation}
for any $\varphi \in \W^{1, \infty}(\Omega)$ satisfying $\varphi \geq 0$.

\end{itemize}

\subsection{Existence of weak solutions}

We are ready to state our main result on the existence of weak solutions to the
Navier--Stokes--Fourier system.

\begin{Theorem} \label{TM1}

Let $\Omega \subset \R^2$ be a bounded domain of class $\C^{2 + \nu}$.
Let the data $\vc{f}$, $G$, $\Ov{\vt}$,
and $\Ov{\vu}$ satisfy the hypotheses \eqref{M1} and \eqref{M4bis}.
Let the pressure $p$ and the internal energy $e$ satisfy \eqref{M2} and
\eqref{M3}. Let the transport coefficients $\mu$, $\lambda$, and $\kappa$
be continuously differentiable functions of $\vt$ satisfying
\eqref{M4}, with
\begin{equation}
0 \leq \alpha < 1.
\end{equation}
Then for any $M > 0$, the Navier--Stokes--Fourier system \eqref{W1}--\eqref{W5}
admits a solution $[\vr, \vu, \vt]$ satisfying
\begin{equation}
M = \intO{\vr}.
\end{equation}
The solution belongs to the class:
\begin{equation*}
\begin{aligned}
\vr &\in \L^\infty(\Omega), \quad 0 \leq \vr < \Ov{\vr} \ \mbox{a.e in} \ \Omega,\\
\vu &\in \W^{1,r}(\Omega; \R^2) \ \mbox{for any}\ 1 \leq r < 2, \\
\vt^\beta &\in \W^{1,r}(\Omega) \ \mbox{for any}\ 1 \leq r < 2, \ \beta \in \R,\quad \vt > 0
\ \mbox{a.e in}\ \Omega,\quad \log(\vt) \in \W^{1,2}(\Omega).
\end{aligned}
\end{equation*}
\end{Theorem}
The rest of the paper is devoted to the proof of Theorem \ref{TM1}.

\section{Approximate system}
\label{A}

Let us define the following cut--off function:
\begin{equation*}
T \in \C^\infty(\R),\quad T(\vr)
\begin{cases}
=\vr & \mbox{if}\ 0 \leq \vr \leq  \Ov{\vr},\\
\in\,]-2\bar{\vr},2\bar{\vr}[ & \mbox{otherwise}.
\end{cases}
\end{equation*}
The solutions will be obtained through a multi--level approximate system including
regularization of various types:
\begin{subequations}
\label{A1-2}
\begin{align}
\label{A1}
&\Div (\vr \vu) = \ep \Delta \vr - \ep( \vr - \vr_M),\quad \Grad \vr \cdot \vc{n} _{|_{\partial \Omega}} = 0,\\
\label{A2}
&\intO{ \Bigl[ \frac{1}{2} T(\vr) ( \vu \otimes \vu : \Grad \bfphi -
\vu \cdot \Grad \vu \cdot \bfphi) + p_R (\vr, \vt)  \Div \bfphi \Bigr] }
= \intO{ \Bigl[ \mathbb{S}_\delta (\vt, \Grad \vu) : \Grad \bfphi - T(\vr) \vc{f} \cdot \bfphi \Bigr] }
\end{align}
\end{subequations}
for any $\bfphi \in X_N$, $(\vu - \vuB) \in X_N$;
\begin{subequations}
\label{A3-4}
\begin{align}
\label{A3}
&\Div (\vr e_R(\vr, \vt) \vu ) + \Div \vQ_\delta(\vt, \Grad \vt) =
\mathbb{S}_\delta (\vt, \vu) : \Grad \vu - p_R(\vr, \vt) \Div \vu
+ \vr (G + \ep),\\&
\label{A4}
\vQ_{\delta} \cdot \vc{n}_{|_{\partial \Omega}}
= L [ \vt -  \Ov{\vt}],
\end{align}
\end{subequations}
with
\begin{equation}
\label{A5}
\mathbb{S}_\delta (\vt, \nabla\vu)
= \mu_\delta (\vt) \bigl( \Grad \vu + (\Grad \vu)^{\tra} - \Div \vu \mathbb{I} \bigr) +
\lambda_\delta (\vt) \Div \vu \mathbb{I}\quad\text{and}\quad
\vQ_\delta (\vt, \Grad \vt) = - \kappa_\delta (\vt) \Grad \vt.
\end{equation}
Furthermore, we have
\begin{equation}
\label{A5b}
\{\vr\}_m=\vr_M.
\end{equation}
There are four parameters: the dimension on the Galerkin
approximation space $N$, the artificial viscosity (mass transport)
coefficient $\ep$, the truncation parameter $R$, and perturbations
by regularizing terms depending on $\delta$.

The specific form of the convective term in \eqref{A2} is borrowed
form Novotn\' y and Pokorn\' y \cite{NovPok11NSFG}.
Moreover, using the same method
as in \cite{NovPok11NSFG} we can show that the approximate system \eqref{A1-2}--\eqref{A5}
admits regular (strong) solution whenever
\begin{subequations}
\begin{align}
&\ep > 0,\quad N < \infty,\quad R < \infty,\quad \delta > 0,\\&
\label{H1}
\kappa_\delta(\vt) = \kappa(\vt) + \delta( \vt^\beta + \vt^{-1}),\quad
\beta \geq 2.
\end{align}
\end{subequations}
and also
\begin{equation}
\label{H4}
\mu_\delta(\vt) \eqldef \mu(\vt) + \delta \vt^a\quad\text{and}\quad
\lambda_\delta(\vt) \eqldef \lambda(\vt) + \delta \vt^a
\ \mbox{for some}\ \alpha < a < 1.
\end{equation}
Moreover there exist \(\underline{\vr}\) and \( \underline{\vt}\) (depending on
\(\varepsilon\), \(N\), \(R\) and \(s\)) such that
\begin{equation}
\label{A6}
0 < \underline{\vr} \leq \vr\quad\text{and}\quad 0 < \underline{\vt}\leq \vt\ \mbox{in}\ \Omega,
\end{equation}
see \cite{CiFeJaPe}, and \cite{FeNo18} for details. The truncated pressure is defined as
\begin{subequations}
\label{H3}
\begin{align}
p_R &\eqldef \vr \vt h_R(\vr)\quad\text{and}\quad e_R \eqldef e \eqldef c_v \vt,\\
h_R(\vr) &\eqldef
\begin{cases}
h(\vr)  &\mbox{if}\ 0 \leq \vr < \Ov{\vr} - \frac{1}{R},\\
h \bigl( \Ov{\vr} - \frac{1}{R} \bigr) + h'\bigl( \Ov{\vr} - \frac{1}{R} \bigr)
\bigl( \vr - \bigl( \Ov{\vr} - \frac{1}{R} \bigr) \bigr) & \mbox{otherwise.}
\end{cases}
\end{align}
\end{subequations}
Accordingly, we may define the entropy $s_R$ by the following identity
\begin{equation*}
s_R(\vr,\vt)\eqldef c_r\log(\vt)-\int_{\vr_{M}}^{\vr}\frac{h_R(z)}{z}\,\mathrm{d}z,
\end{equation*}
such that the following Gibbs' relation is satisfied
\begin{equation} \label{H2}
\vt D s_{R} = D e_R + p_R D \Bigl( \frac{1}{\vr} \Bigr).
\end{equation}

\section{Uniform bounds on approximate solutions}
\label{U}

Our goal is to establish uniform bounds for the approximate solutions solving
\eqref{A1-2}--\eqref{A5}.

\subsection{Entropy equation}

As $\vt > 0$, the internal energy balance \eqref{A3} divided on $\vt$ reads:
\begin{equation*}
\begin{aligned}
&\frac{1}{\vt} \Div (\vr \vu) e_R(\vr, \vt) + \frac{1}{\vt} \vr \Grad e_R \cdot \vu
- \frac{1}{\vt} \Div ( \kappa_\delta (\vt) \Grad \vt)
\\&= \frac{1}{\vt} \mathbb{S}_\delta(\vt, \Grad \vu): \Grad \vu -
\frac{1}{\vt} p_R(\vr, \vt) \Div \vu + \frac{\vr}{\vt}(G + \ep),
\end{aligned}
\end{equation*}
in other words,  we have
\begin{equation}
\label{A7}
\begin{aligned}
&\frac{1}{\vt} \Bigl[ \mathbb{S}_\delta (\vt, \Grad \vu) : \Grad \vu + \frac{ \kappa_\delta (\vt) }{\vt}
|\Grad \vt |^2 \Bigr] + \frac{\vr}{\vt}(G + \ep) \\ &= \Div \Bigl( \frac{\vQ_\delta (\vt, \Grad \vt) }{\vt} \Bigr)
+ \frac{1}{\vt} \Div (\vr \vu) e_R(\vr, \vt) + \frac{1}{\vt} \vr \Grad e_R \cdot \vu
+ \frac{1}{\vt} p_R(\vr, \vt) \Div \vu \\
&= \Div \Bigl( \frac{\vQ_\delta (\vt, \Grad \vt) }{\vt} \Bigr)
 + \vr \Bigl( \frac{1}{\vt} \Grad e_R
 - \frac{1}{\vr^2 \vt} p_R(\vr, \vt)\Grad \vr \Bigr) \cdot \vu
+\frac{1}{\vt} \Div (\vr \vu) \Bigl( e_R(\vr, \vt) + \frac{p_R(\vr, \vt)}{\vr} \Bigr).
\end{aligned}
\end{equation}
In view of hypothesis \eqref{H2}, the relation \eqref{A7} can be written in the following form:
\begin{equation}
\begin{aligned}
&\frac{1}{\vt} \Bigl[ \mathbb{S}_\delta (\vt, \Grad \vu) : \Grad \vu + \frac{ \kappa_\delta (\vt) }{\vt}
|\Grad \vt |^2 \Bigr] + \frac{\vr}{\vt}(G + \ep)\\
&= \Div \Bigl( \frac{\vQ_\delta (\vt, \Grad \vt) }{\vt}\Bigr)
 + \Div( \vr  s_R(\vr, \vt)  \vu) + \Div (\vr \vu) \Bigl( \frac{e_R(\vr, \vt)}{\vt} -
s_R(\vr, \vt) + \frac{p_R(\vr, \vt)}{\vr \vt} \Bigr).
\end{aligned}
\end{equation}
Finally, by using \eqref{A1}, we may conclude
\begin{equation}
\label{A8}
\begin{aligned}
&\frac{1}{\vt} \Bigl[ \mathbb{S}_\delta (\vt, \Grad \vu) : \Grad \vu + \frac{ \kappa_\delta (\vt) }{\vt}
|\Grad \vt |^2 \Bigr] + \frac{\vr}{\vt}(G + \ep)
= \Div \Bigl( \frac{\vQ_\delta (\vt, \Grad \vt) }{\vt} \Bigr)
 + \Div \Big( \vr  s_R(\vr, \vt)  \vu \Big)\\&
+ \ep( \Delta \vr - ( \vr - \vr_M)) \Bigl( \frac{e_R(\vr, \vt)}{\vt} -
s_R(\vr, \vt) + \frac{p_R(\vr, \vt)}{\vr \vt} \Bigr).
\end{aligned}
\end{equation}
The desired uniform bounds follow by integrating \eqref{A8} over $\Omega$, we get
\begin{equation}
\label{A9}
\begin{aligned}
&\intO{\frac{1}{\vt} \Bigl[ \mathbb{S}_\delta (\vt, \Grad \vu) : \Grad \vu + \frac{ \kappa_\delta (\vt) }{\vt}
|\Grad \vt |^2 \Bigr]}+ L \int_{\partial \Omega} \frac{\Ov{\vt}}{\vt} \ {\rm dS}
+ \ep \intO{ (\vr - \vr_M)\frac{p_R(\vr, \vt)}{\vr \vt} } \\&+ \intO{ \frac{\vr}{\vt}(G + \ep)}
= L |\partial \Omega|
+ \ep \intO{ \Delta \vr  \Bigl( \frac{e_R(\vr, \vt)}{\vt} -
s_R(\vr, \vt) + \frac{p_R(\vr, \vt)}{\vr \vt}\Bigr)}\\&
- \ep \intO{ ( \vr - \vr_M) \Bigl( \frac{e_R(\vr, \vt)}{\vt} -
s_R(\vr, \vt)  \Bigr)}.
\end{aligned}
\end{equation}
Furthermore, using Gibbs' relation \eqref{H2} and the boundary condition
$\Grad \vr \cdot \vc{n}_{|_{\partial \Omega}} = 0$, we find
\begin{equation*}
\begin{aligned}
&\ep \intO{ \Delta \vr  \Bigl( \frac{e_R(\vr, \vt)}{\vt} -
s_R(\vr, \vt) + \frac{p_R(\vr, \vt)}{\vr \vt}\Bigr) }=
- \ep \intO{ \Grad \vr  \cdot \Grad \Bigl( \frac{e_R(\vr, \vt)}{\vt} -
s_R(\vr, \vt) + \frac{p_R(\vr, \vt)}{\vr \vt}\Bigr) }\\
&= - \ep \intO{ \frac{1}{\vr \vt} \frac{\partial p_R}{\partial \vr}(\vr, \vt) |\Grad \vr|^2 }
+ \ep \intO{ \frac{1}{\vt^2} \Bigl[ e_R(\vr, \vt)
+\frac{p_R(\vr, \vt)}{\vr} - \frac{\vt}{\vr} \frac{\partial p_R}{\partial \vt}(\vr, \vt)
\Bigr]  \Grad \vr \cdot \Grad \vt }.
\end{aligned}
\end{equation*}
Consequently, relation \eqref{A9} gives rise to
\begin{equation}
\label{A10}
\begin{aligned}
&\intO{\frac{1}{\vt} \Bigl[ \mathbb{S}_\delta (\vt, \Grad \vu) : \Grad \vu + \frac{ \kappa_\delta (\vt) }{\vt}
|\Grad \vt |^2 \Bigr]} + L \int_{\partial \Omega} \frac{\Ov{\vt}}{\vt} \ {\rm dS}
+ \ep \intO{ (\vr - \vr_M)  \ \frac{p_R(\vr, \vt)}{\vr \vt} } \\
&+  \ep \intO{ \frac{1}{\vr \vt} \frac{\partial p_R}{\partial \vr}(\vr, \vt) |\Grad \vr|^2 }
+ \intO{ \frac{\vr}{\vt} (G + \ep) }\\&= L |\partial \Omega|
+ \ep \intO{ \Bigl[ \frac{e_R(\vr, \vt)}{\vt}
+ \frac{p_R(\vr, \vt)}{\vr \vt} - \frac{1}{\vr} \frac{\partial p_R}{\partial \vt}(\vr, \vt)
\Bigr]  \Grad \vr \cdot \Grad \log( \vt ) }\\&
- \ep \intO{ ( \vr - \vr_M) \Bigl[  \frac{e_R(\vr, \vt)}{\vt} -
s_R(\vr, \vt)\Bigr] }.
\end{aligned}
\end{equation}
Thus, using the specific form \eqref{H3} of the pressure and energy truncations, we may infer that
\begin{equation}
\label{A11}
\begin{aligned}
&\intO{\frac{1}{\vt}\Bigl[\mathbb{S}_\delta (\vt, \Grad \vu) : \Grad \vu + \frac{ \kappa_\delta (\vt) }{\vt}
|\Grad \vt |^2 \Bigr]} + L \int_{\partial \Omega} \frac{\Ov{\vt}}{\vt} \ {\rm dS}
+ \ep \intO{ (\vr - \vr_M) h_R (\vr) }\\& +  \ep \intO{ \Bigl[ h_R'(\vr)+
\frac{h_R(\vr)}{\vr} \Bigr] |\Grad \vr|^2 }
+ \intO{ \frac{\vr}{\vt} (G + \ep) }
= L |\partial \Omega|
+ \ep c_v \intO{   \Grad \vr \cdot \Grad \log( \vt ) }\\&
-\ep \intO{ ( \vr - \vr_M) \Bigl(
\int_{\vr_M}^\vr \frac{ h_R(z) }{z} \ {\rm d}z \Bigr)}
+\ep c_v\intO{(\vr-\vr_M)(\log(\vt)-{\{\log(\vt)\}}_m)}.
\end{aligned}
\end{equation}
Note that, in accordance with \eqref{H2}, the entropy reads
\begin{equation}
\label{A12}
s(\vr, \vt) = c_v \log (\vt) - s_{R,h}(\vr)\quad\text{and}\quad
s'_{R,h} (\vr) = \frac{1}{\vr} h_R(\vr).
\end{equation}
We determine now some uniform bounds based on the entropy. First observe that
\begin{equation*}
h_R'(\vr) + \frac{h_R(\vr)}{\vr} \geq c_D > 0 \ \mbox{independently of}\ R.
\end{equation*}
Returning to \eqref{A11}, we may deduce that
\begin{equation}
\label{A13}
\begin{aligned}
&\intO{\frac{1}{\vt}\Bigl[ \mathbb{S}_\delta (\vt, \Grad \vu) : \Grad \vu + \frac{ \kappa_\delta (\vt) }{\vt}
|\Grad \vt |^2 \Bigr]} + L \int_{\partial \Omega} \frac{\Ov{\vt}}{\vt} \ {\rm dS}  +
\ep \intO{ ( \vr - \vr_M) \Bigl( \int_{\vr_M}^\vr \frac{h_R(z)}{z} \Bigr) {\rm d}z  }\\
&+ \ep \intO{ (\vr - \vr_M)( h_R (\vr) - h_R(\vr_M))}
+  \ep \intO{ \Bigl[ h_R'(\vr) + \frac{h_R(\vr)}{\vr} \Bigr] |\Grad \vr|^2 }
+ \intO{ \frac{\vr}{\vt} (G + \ep) } \leq c_D,
\end{aligned}
\end{equation}
where the constant $c_D$ is independent of the parameters $N$, $R$, $\ep$, and $\delta$.
Notice that
\begin{equation*}
(\vr-\vr_M)\int_{\vr}^{\vr_M}\frac{h(z)}{z}\ {\rm d}z\geq 0
\quad\text{and}\quad
\intO{ (\vr - \vr_M) ( h_R (\vr) - h_R(\vr_M))}\geq 0.
\end{equation*}

\subsection{Trudinger--Moser inequality}
\label{TM}

In this section, we derive rather strong bounds on the temperature based on the
Trudinger--Moser inequality.

\subsubsection{Bounds on the temperature gradient}

Integrating \eqref{A3}  over \(\Omega\) and using \eqref{A5b} and \eqref{M1}, we find
\begin{equation*}
\begin{aligned}
L \int_{\partial \Omega} (\vt - \Ov{\vt} ) {\rm dS}&=
\intO{ \Big[ \mathbb{S}_\delta (\vt, \Grad \vu) : \Grad \vu - p_R(\vr, \vt) \Div \vu \Big] }
+ \intO{\vr (G + \ep)}  \\ &\leq  \intO{\Bigl[ \mathbb{S}_\delta (\vt, \Grad\vu) : \Grad \vu -
p_R(\vr, \vt) \Div \vu\Bigr] }
+ c_D.
\end{aligned}
\end{equation*}
Next, taking $\vu - \Ov{\vu}$ as a test function in the momentum equation \eqref{A2}, we get
\begin{equation*}
\begin{aligned}
&\intO{ \Big[ \mathbb{S}_\delta (\vt, \Grad \vu) : \Grad \vu -
p_R (\vr, \vt)  \Div \vu   \Big]}\\
&= \intO{ \Bigl[ \frac{1}{2} T(\vr) (
\vu \cdot \Grad \vu \cdot \Ov{\vu} - \vu \otimes \vu : \Grad \Ov{\vu}) +
\mathbb{S}_\delta (\vt, \Grad \vu) : \Grad \Ov{\vu} + T(\vr) \vc{f} \cdot (\vu - \Ov{\vu})
\Bigr] },
\end{aligned}
\end{equation*}
and, consequently, we have
\begin{equation}
\label{A14}
\begin{aligned}
L \int_{\partial \Omega} (\vt - \Ov{\vt} ) {\rm dS}
&\leq \intO{ \Bigl[ \frac{1}{2} T(\vr) (
\vu \cdot \Grad \vu \cdot \Ov{\vu} - \vu \otimes \vu : \Grad \Ov{\vu}) +
\mathbb{S}_\delta (\vt, \Grad \vu) : \Grad \Ov{\vu} + T(\vr) \vc{f} \cdot \vu )
\Bigr] }
+ c_{D}\\
&\aleq \Bigl( 1 + \intO{ \Bigl[ | \vu \cdot \Grad \vu | + |\vu|^2 +
|\mathbb{S}_\delta (\vt, \Grad \vu)| \Bigr] } \Bigr).
\end{aligned}
\end{equation}
Our goal is to show that all the integrals on the right--hand side of \eqref{A14} can be controlled by
$\| \vt \|^\ell_{\L^q(\Omega)}$ for some $q< \infty$ and $\ell< 1$. Let us first  observe that
\begin{equation*}
\intO{ |(\vu -\Ov{\vu} ) \cdot \Grad (\vu - \Ov{\vu}) | }
\leq \intO{ |\vu - \Ov{\vu}| \ |\Grad \vu - \Grad \Ov{\vu}|}.
\end{equation*}
Then the Korn's inequality leads to
\begin{equation}
\label{korn}
\| \Grad \vu - \Grad \Ov{\vu} \|_{\L^r(\Omega)} \leq c_K(r)
\| \mathbb{D} \vu - \mathbb{D} \Ov{\vu} \|_{\L^r(\Omega)},\ 1 < r < \infty,
\end{equation}
and the Poincar\' e's inequality  gives
\begin{equation}
\label{porn}
\| \vu - \Ov{\vu} \|_{\L^{s_1}(\Omega)} \leq c_P \| \Grad \vu - \Grad \Ov{\vu}
\|_{\L^r(\Omega)},\ s_1 \leq \frac{2r}{2 - r}\text{ and } s_1=\frac{r}{r-1}.
\end{equation}
Consequently, there is $\frac43 < r < 2$ such that
\begin{equation*}
\intO{ |(\vu -\Ov{\vu} ) \cdot \Grad (\vu - \Ov{\vu}) | } \aleq
\| \mathbb{D} \vu - \mathbb{D} \Ov{\vu} \|_{\L^r(\Omega)}^2,
\end{equation*}
whence
\begin{equation}
\label{A15}
\intO{ |\vu \cdot \Grad \vu | } \aleq (1 + \| \mathbb{D} \vu \|_{\L^r(\Omega;\R^4)}^2) \ \mbox{for some}
\ \frac43< r < 2.
\end{equation}
On the other hand, the entropy estimates \eqref{A13} together with hypothesis \eqref{H4} yield to
\begin{equation}
\label{zz}
\Bigl\|\frac1{\sqrt\vt}\mathbb{D}\vu\Bigr\|^2_{\L^2(\Omega; \R^4)}+
\| \sqrt{\delta} \vt^{\frac{a - 1}{2}} \mathbb{D} \vu\|^2_{\L^2(\Omega; \R^4)} \leq c_{D}.
\end{equation}
According to H\" older's inequality, we obatin
\begin{equation}
\label{A16}
\| \mathbb{D} \vu \|_{\L^r(\Omega;\R^4)} = \frac{1}{\sqrt{\delta}}\| \vt^{\frac{1- a}{2}}
\sqrt{\delta} \vt^{\frac{a-1}{2}} \mathbb{D} \vu \|_{\L^r(\Omega;\R^4)}
\aleq \frac{1}{\sqrt{\delta}} \| \vt^{\frac{1- a}{2}}\|_{\L^s(\Omega)}
\aleq \frac{1}{\sqrt{\delta}} \| \vt\|_{\L^s(\Omega)}^{\frac{1- a}{2}}, \
\frac{1}{s} + \frac{1}{2} = \frac{1}{r}.
\end{equation}
Combining \eqref{A15} and \eqref{A16}, we get
\begin{equation}
\label{A18}
\intO{ |\vu \cdot \Grad \vu | } \leq c(\delta)\bigl(1 + \| \vt\|_{\L^s(\Omega)}^{{1- a}}\bigr)
\ \mbox{for some}\ 1 < s < \infty.
\end{equation}
Similarly, we deduce
\begin{equation}
\label{A19}
\intO{ |\vu  |^2 } \leq c(\delta)\bigl(1 + \| \vt\|_{\L^s(\Omega)}^{{1- a}}\bigr)
\ \mbox{for some}\ 1 < s < \infty.
\end{equation}
Finally, using once again the entropy bound \eqref{A13} as well as \eqref{zz},
the last integral in \eqref{A14} can be controlled as follows:
\begin{equation}
\label{A20}
\begin{aligned}
\intO{ |\mathbb{S}_\delta (\vt, \Grad \vu) | } &=
\frac{1}{\sqrt \delta} \Bigl\| \Bigl( \frac{1 + \vt^a}{\vt} \Bigr)^{-\frac{1}{2}}
\sqrt{\delta} \Bigl( \frac{1 + \vt^a}{\vt} \Bigr)^{\frac{1}{2}} \mathbb{D} \vu  \Bigr\|_{\L^1(\Omega; \R^4)} \\
&\leq c_1(\delta) \bigl( 1 + \| \vt^{\frac{1 - a}{2}}\|_{\L^2(\Omega)} \bigr)
\leq c_2(\delta) \bigl( 1 + \| \vt \|_{\L^2(\Omega)}^{\frac{1 - a}{2}} \bigr).
\end{aligned}
\end{equation}
Summing up \eqref{A14}--\eqref{A20}, we obtain
\begin{equation}
\label{A21}
\int_{\partial \Omega} \vt \ {\rm dS} \leq c(\delta)\bigl(1 + \| \vt \|^{\ell}_{\L^s(\Omega)}\bigr)
\ \mbox{for some}\ \ell < 1, \ 1<s < \infty.
\end{equation}
Seeing that
\begin{itemize}

\item In accordance with hypothesis \eqref{H1} and the entropy estimates \eqref{A13}, we get
\begin{equation*}
\| \Grad \vt \|_{\L^2(\Omega; \R^2)} \leq c(\delta).
\end{equation*}

\item The Poincare's inequality leads to
\begin{equation*}
\| \vt \|_{\W^{1,2}(\Omega)} \aleq \| \Grad \vt \|_{\L^2(\Omega; \R^2)} + \| \vt \|_{\L^1(\partial \Omega)}.
\end{equation*}

\item The Sobolev embedding gives
\begin{equation*}
\| \vt \|_{\L^s(\Omega)} \leq c(s) \| \vt \|_{\W^{1,2}(\Omega)} \ \mbox{for any}\ 1<s < \infty.
\end{equation*}

\item It comes from \eqref{A13} and \eqref{A21} that
\begin{equation}
\label{A22}
\| \vt \|_{\W^{1,2}(\Omega)} \leq c(\delta)\quad\text{and}\quad
\| \log(\vt)\|_{\W^{1,2}(\Omega)} \leq c(\delta).
\end{equation}

\end{itemize}

\subsubsection{Estimates of $\vt$ near absolute zero}

We apply the Trudinger--Moser inequality, more specifically the Sobolev embedding
\begin{equation*}
\W^{1,2}(\Omega) \hookrightarrow \L_\Phi (\Omega) ,
\quad\Omega \subset \R^2 \ \mbox{a bounded Lipschitz domain,}
\end{equation*}
where $\L_{\Phi}$ is the Orlicz space with the generating function
$\Phi(z) = \exp (z^2) - 1$, see e.g. Adams \cite[Chap.~8 ]{Adams75SS}.
In particular,
\begin{equation*}
\| v \|_{\L_\Phi(\Omega)} \leq c_S \| v \|_{\W^{1,2}(\Omega)},
\end{equation*}
where $\| \cdot \|_{\L_\Phi(\Omega)}$ is the associated Luxemburg norm. In particular,
\begin{equation} \label{A23}
\intO{ \exp\Bigl( \frac{ |v|^2 }{ c_S^2 \| v \|^2_{\W^{1,2}(\Omega) } } \Bigr) - 1  }
\leq \intO{ \exp\Bigl( \frac{ |v|^2 }{ \| v \|^2_{\L_\Phi(\Omega) } }\Bigr) - 1  } \leq 1
\ \mbox{for any}\ v \in \W^{1,2}(\Omega),\ v \not\equiv 0.
\end{equation}
In view of the bounds established in \eqref{A22}, we
may apply \eqref{A23} to $v = \log(\vt)$ which leads to
\begin{equation}
\label{A24}
\intO{\exp\bigl( { \omega (\delta) |\log(\vt)|^2 } \bigr)}
\leq 1 + |\Omega| \ \mbox{for certain} \ \omega(\delta) > 0,
\end{equation}
and it follows that
\begin{equation}
\label{A25}
\| \vt^\beta \|_{\L^s(\Omega)} \leq c(s, \beta, \delta) \ \mbox{for any}\ 1 \leq s < \infty, \ \beta \in \R.
\end{equation}
Putting together \eqref{A13}, \eqref{A22} and \eqref{A25},
the following uniform bounds for the approximate solutions
depending only on the parameter $\delta > 0$ is obtained:
\begin{equation}
\label{A26}
\begin{aligned}
\| \vt \|_{\W^{1,2}(\Omega)} + \| \log (\vt) \|_{\W^{1,2}(\Omega)} &\leq c(\delta),\\
\| \vt^\beta \|_{\L^s(\Omega)} &\leq c(\delta, \beta, s) \ \mbox{for any}\ 1\leq s < \infty,\ \beta \in \R,\\
\| \vu \|_{\W^{1,r}(\Omega)} &\leq c(\delta, r), \ 1 \leq r < 2.
\end{aligned}
\end{equation}
In addition, we record the bounds on the density can be deduced from
the entropy bounds \eqref{A13}, the standard elliptic estimates applied
to the approximate equation of continuity \eqref{A1} that depend on $\ep$ and
the fact that \(\vu\) is bounded in \(\W^{1,r}(\Omega)\), namely we have
\begin{equation}
\label{A27}
\| \vr \|_{\W^{2,q}(\Omega)} \leq c(\ep, \delta, q) \ \mbox{for any}\ 1 \leq q < 2.
\end{equation}

\section{Limit $N \to \infty$}
\label{N}

Keeping $R > 0$, $\ep > 0$, $\delta > 0$ fixed, we perform the limit $N \to \infty$.
Let $\{ \vr_N, \vu_N, \vt_N \}_{N \geq 1}$ be a sequence of solutions to \eqref{A1-2}--\eqref{A5}.
In view of the uniform estimates summarized in
\eqref{A26} and \eqref{A27} and passing to the subsequences, if necessary, we find
\begin{subequations}
\label{N1}
\begin{align}
\vr_N &\rightharpoonup \vr \ \mbox{weakly in}\ \W^{2,q}(\Omega)\ \mbox{for any}\ 1\leq q < 2,\\
\vt_N &\rightharpoonup \vt \ \mbox{weakly in}\ \W^{1,2}(\Omega),\\
\vt^\beta_N &\rightharpoonup \vt^\beta \ \mbox{weakly in}\ \W^{1,r}(\Omega) \text{ and }
\vt^\beta_N \to \vt^\beta \ \mbox{in}\ \L^s(\Omega)
\ \mbox{for any}\ \beta \in \R,\ 1 \leq r < 2,\ 1 \leq s < \infty, \\
\vu_N &\rightharpoonup \vu \ \mbox{weakly in}\ \W^{1,r}(\Omega; \R^2) \ \mbox{for any}\ 1 \leq r < 2,
\end{align}
\end{subequations}
as $N \to \infty$.

\subsection{Equation of continuity}

According to \eqref{N1} and some standard compactness arguments for
Sobolev spaces, we easily deduce
\begin{equation}
\label{N2}
\Div (\vr \vu) = \ep \Delta \vr - \ep (\vr - \vr_M), \quad \Grad \vr \cdot \vc{n}_{|_{\partial \Omega}} = 0,
\quad
\vr \geq 0, \quad \frac{1}{|\Omega|} \intO{ \vr} = \vr_M.
\end{equation}
As uniform bounds on the velocity gradient are no longer available, the limit density may
vanish at certain point of $\Omega$.

\subsection{Momentum balance}

Using similar arguments as above,
we perform the limit $N \to \infty$ in the momentum equation \eqref{A2}, we find
\begin{equation}
\label{N3}
\begin{aligned}
&\intO{ \Bigl[ \frac{1}{2} T(\vr) \big( \vu \otimes \vu : \Grad \bfphi -
\vu \cdot \Grad \vu \cdot \bfphi \big) + p_R (\vr, \vt)  \Div \bfphi \Bigr] }
\\& = \intO{ \Bigl[ \mathbb{S}_\delta (\vt, \Grad \vu) : \Grad \bfphi - T(\vr) \vc{f} \cdot \bfphi \Bigr] }
\end{aligned}
\end{equation}
for any $\bfphi \in \W^{1,q}_0(\Omega; \R^2)$, $q > 2$, $(\vu - \vuB) \in \W^{1,r}_0(\Omega; \R^2)$,
$1 \leq r < 2$.
Notice that the relation \eqref{N3} is first obtained for any $\bfphi \in X_M$,
$M$ fixed and then it is extended to all $\bfphi \in \W^{1,q}_0(\Omega;\R^2)$
by using a density argument.

\subsection{Total energy balance}

The arguments giving rise to the total energy balance are much more delicate.
The first step consists to pass to the weak formulation of
the internal energy equation \eqref{A3} and \eqref{A4} to get
\begin{equation}
\label{N4}
\begin{aligned}
&-\intO{ \vr_N e_R(\vr_N, \vt_N) \vu_N \cdot \Grad \varphi  } + \int_{\partial \Omega} L (\vt_N - \Ov{\vt})
\varphi {\rm dS} + \intO{ \kappa_\delta(\vt_N) \Grad \vt_N \cdot \Grad \varphi } \\&=
\intO{ \Big[ \mathbb{S}_\delta (\vt_N, \Grad \vu_N) : \Grad \vu_N - p_R(\vr_N, \vt_N)
\Div \vu_N \Big] \varphi }
+ \intO{ \vr_N (G + \ep) \varphi }
\end{aligned}
\end{equation}
for any $\varphi \in \W^{1,\infty}(\Omega)$.
Then $\vu_N - \Ov{\vu} \in X_N$ is used as a test function in the approximate
momentum equation \eqref{A2} giving
\begin{equation}
\label{N5}
\begin{aligned}
&\intO{ \Bigl[ \frac{1}{2} T(\vr_N)
(\vu_N \cdot \Grad \vu_N \cdot \Ov{\vu}  - \vu_N \otimes \vu_N : \Ov{\vu} ) \Bigr] }
\\& = \intO{ \Big[ \mathbb{S}_\delta (\vt_N, \Grad \vu_N) : \Grad (\vu_N - \Ov{\vu}) -
p_R (\vr_N, \vt_N)  \Div \vu_N \Big] }
 - \intO{ T(\vr_N) \vc{f} \cdot (\vu_N - \Ov{\vu})}.
\end{aligned}
\end{equation}
Taking the sum of \eqref{N5} with \eqref{N4} in the case where $\varphi = -1$, we obtain
\begin{equation*}
\begin{aligned}
&\int_{\partial \Omega} L (\Ov{\vt} - \vt_N) \ {\rm dS}+
\intO{ \Bigl[ \frac{1}{2} T(\vr_N) (\vu_N \cdot \Grad \vu_N \cdot \Ov{\vu}
- \vu_N \otimes \vu_N : \Ov{\vu}) \Bigr] }
\\ &= - \intO{ \mathbb{S}_\delta (\vt_N, \Grad \vu_N) : \Grad\Ov{\vu} } - \intO{ \vr_N (G + \ep ) }
 - \intO{ T(\vr_N) \vc{f} \cdot (\vu_N - \Ov{\vu})}.
\end{aligned}
\end{equation*}
Letting $N \to \infty$ we may therefore infer that
\begin{equation}
\label{N6}
\begin{aligned}
&\int_{\partial \Omega} L (\Ov{\vt} - \vt) \ {\rm dS} +
\intO{ \Bigl[ \frac{1}{2} T(\vr)(\vu \cdot \Grad \vu \cdot \Ov{\vu}  - \vu \otimes \vu : \Ov{\vu}) \Bigr] }
\\ &= - \intO{ \mathbb{S}_\delta (\vt, \Grad \vu) : \Grad\Ov{\vu} } - \intO{ \vr (G + \ep) }
 - \intO{ T(\vr_N) \vc{f} \cdot (\vu - \Ov{\vu})  }.
\end{aligned}
\end{equation}
We consider now the internal energy \eqref{N4} and we assume that the test function
$\psi \in \W^{1,\infty}(\Omega)$ with $\psi \geq 0$, we find
\begin{equation*}
\begin{aligned}
&\intO{ \mathbb{S}_\delta (\vt_N, \Grad \vu_N) : \Grad \vu_N \psi }
=-\intO{ \vr_N e_R(\vr_N, \vt_N) \vu_N \cdot \Grad \psi  } + \int_{\partial \Omega} L (\vt_N - \Ov{\vt})
\psi {\rm dS}\\&+ \intO{ \kappa_\delta(\vt_N) \Grad \vt_N \cdot \Grad \psi }
+ \intO{ p_R(\vr_N, \vt_N) \Div \vu_N \psi } - \intO{ \vr_N (G + \ep) \psi }.
 \end{aligned}
\end{equation*}
Letting $N \to \infty$ and using the weak lower semi--continuity of convex functions, we obtain
\begin{equation}
\label{N7}
\begin{aligned}
&\intO{ \mathbb{S}_\delta (\vt, \Grad \vu) : \Grad \vu \psi }
\leq -\intO{ \vr e_R(\vr, \vt) \vu \cdot \Grad \psi  } + \int_{\partial \Omega} L (\vt - \Ov{\vt})
\psi {\rm dS}\\&+ \intO{ \kappa_\delta(\vt) \Grad \vt \cdot \Grad \psi }
+ \intO{ p_R(\vr, \vt) \Div \vu \psi } - \intO{ \vr (G + \ep) \psi }
 \end{aligned}
\end{equation}
for any $\psi \in \W^{1, \infty}(\Omega)$, $\psi \geq 0$. In particular, we have
\begin{equation}
\label{N8}
\mathbb{S}_\delta (\vt, \Grad \vu) : \Grad \vu \in \L^1(\Omega).
\end{equation}
The bound \eqref{N8} allows us to consider the momentum balance \eqref{N3}
with the test function
\begin{equation*}
\bfphi \eqldef \varphi (\vu - \Ov{\vu})\text{ with }\varphi \in \C^1(\Ov{\Omega}),
\end{equation*}
yielding to the following identity:
\begin{equation*}
\begin{aligned}
&\intO{ \varphi \mathbb{S}_\delta (\vt, \Grad \vu) : \Grad \vu } \\&=
\intO{ \Bigl[ \frac{1}{2} T(\vr)( \vu \otimes \vu : \Grad [\varphi (\vu - \Ov{\vu})] -
\varphi \vu \cdot \Grad \vu \cdot (\vu - \Ov{\vu})) + p_R (\vr, \vt)  \Div [\varphi (\vu - \Ov{\vu})] \Bigr] } \\
&+ \intO{\varphi T(\vr) \vc{f} \cdot (\vu -
\Ov{\vu}) } + \intO{ \mathbb{S}_\delta (\vt, \Grad \vu) \cdot \Grad \varphi \cdot (\Ov{\vu} - \vu) }
+ \intO{ \varphi \mathbb{S}_\delta (\vt, \Grad \vu) : \Grad  \Ov{\vu}}\\
&=
\intO{ \frac{1}{2} T(\vr)( \vu \otimes \vu \cdot \Grad \varphi \cdot (\vu - \Ov{\vu}) +
\varphi \vu \cdot \Grad \vu \cdot \Ov{\vu} - \varphi (\vu \otimes \vu) : \Grad \Ov{\vu}) }
\\
&+ \intO{p_R (\vr, \vt)  \Grad\varphi \cdot (\vu - \Ov{\vu})  } +
\intO{ \varphi p_R(\vr, \vt) \Div \vu }\\
 &+ \intO{\varphi T(\vr) \vc{f} \cdot (\vu -
\Ov{\vu}) } + \intO{ \mathbb{S}_\delta (\vt, \Grad \vu) \cdot \Grad \varphi \cdot (\Ov{\vu} - \vu) }
+ \intO{ \varphi \mathbb{S}_\delta (\vt, \Grad \vu) : \Grad  \Ov{\vu}}.
\end{aligned}
\end{equation*}
Thus taking $\varphi = \psi \geq 0$ and using \eqref{N7}, we get
\begin{equation}
\label{N9}
\begin{aligned}
&\intO{ \frac{1}{2} T(\vr)( \vu \otimes \vu \cdot \Grad \psi \cdot (\vu - \Ov{\vu}) +
\psi \vu \cdot \Grad \vu \cdot \Ov{\vu} - \psi (\vu \otimes \vu) : \Grad \Ov{\vu}) }
+ \intO{p_R (\vr, \vt)  \Grad\psi \cdot (\vu - \Ov{\vu})  }\\
 &+ \intO{\psi T(\vr) \vc{f} \cdot (\vu -
\Ov{\vu}) } + \intO{ \mathbb{S}_\delta (\vt, \Grad \vu) \cdot \Grad \psi \cdot (\Ov{\vu} - \vu) }
+ \intO{ \psi \mathbb{S}_\delta (\vt, \Grad \vu) : \Grad  \Ov{\vu}}\\
&\leq
-\intO{ \vr e_R(\vr, \vt) \vu \cdot \Grad \psi  } + \int_{\partial \Omega} L (\vt - \Ov{\vt})
\psi {\rm dS} + \intO{ \kappa_\delta(\vt) \Grad \vt \cdot \Grad \psi } - \intO{ \vr (G + \ep) \psi }
 \end{aligned}
\end{equation}
for any $\psi \in \C^1(\Ov{\Omega})$ with $\psi \geq 0$.
Let us consider now $0 \leq \psi \leq 1$. Taking $(1-\psi) \geq 0$ as a
test function in \eqref{N9}, we obtain
\begin{equation}
\label{N10}
\begin{aligned}
&-\intO{ \frac{1}{2} T(\vr)( \vu \otimes \vu \cdot \Grad \psi \cdot (\vu - \Ov{\vu}) +
\psi \vu \cdot \Grad \vu \cdot \Ov{\vu} - \psi (\vu \otimes \vu) : \Grad \Ov{\vu}) }\\&
- \intO{p_R (\vr, \vt)  \Grad\psi \cdot (\vu - \Ov{\vu})  }- \intO{\psi T(\vr) \vc{f} \cdot (\vu -
\Ov{\vu}) } - \intO{ \mathbb{S}_\delta (\vt, \Grad \vu) \cdot \Grad \psi \cdot (\Ov{\vu} - \vu) }
\\&- \intO{ \psi \mathbb{S}_\delta (\vt, \Grad \vu) : \Grad  \Ov{\vu}}
\leq
\intO{ \vr e_R(\vr, \vt) \vu \cdot \Grad \psi  } - \int_{\partial \Omega} L (\vt - \Ov{\vt})
\psi {\rm dS} - \intO{ \kappa_\delta(\vt) \Grad \vt \cdot \Grad \psi } \\&+\intO{ \vr (G + \ep) \psi }
+ \intO{ \frac{1}{2} T(\vr)( \vu \otimes \vu : \Grad \Ov{\vu} - \vu \cdot \Grad \vu \cdot \Ov{\vu}) }
+ \intO{ T(\vr) \vc{f} \cdot (\Ov{\vu} - \vu) } \\&- \intO{ \mathbb{S}_\delta (\vt, \Grad \vu) : \Grad  \Ov{\vu}}
+\int_{\partial \Omega} L (\vt - \Ov{\vt}) \ {\rm dS} - \intO{ \vr (G + \ep) }.
 \end{aligned}
\end{equation}
However, by virtue of \eqref{N6}, the sum of the integrals on the right--hand side of
\eqref{N10} that do not contain $\psi$ vanishes. Consequently, \eqref{N10} is
reduced to \eqref{N9} with the opposite inequality. Therefore, we conclude \eqref{N9} holds
as an equality, namely we have
\begin{equation}
\label{N11}
\begin{aligned}
&\intO{\Bigl( \frac{1}{2} T(\vr) |\vu|^2 + \vr e_R(\vr, \vt)\Bigr) \vu \cdot \Grad \varphi }
+ \intO{p_R (\vr, \vt) \vu \cdot \Grad\varphi   } + \intO{ \vQ_\delta \cdot \Grad \varphi }
\\&- \intO{ \mathbb{S}_\delta (\vt, \Grad \vu) \cdot \vu \cdot \Grad \varphi  }
= \int_{\partial \Omega} L (\vt - \Ov{\vt})
\varphi \ {\rm dS}  - \intO{ \varphi \vr (G + \ep)  } - \intO{\varphi T(\vr) \vc{f} \cdot \vu}
\\
&+ \intO{ \frac{1}{2} T(\vr) (
 (\vu \otimes \vu) : \Grad (\varphi \Ov{\vu}) - \vu \cdot \Grad \vu \cdot (\varphi \Ov{\vu})) }
+ \intO{ p_R (\vr, \vt) \Div (\varphi \Ov{\vu}) } \\
 &+ \intO{T(\vr) \vc{f} \cdot
(\varphi \Ov{\vu}) } - \intO{ \mathbb{S}_\delta (\vt, \Grad \vu) : \Grad (\varphi \Ov{\vu}) }
 \end{aligned}
\end{equation}
for any $\varphi \in \W^{1,\infty}({\Omega})$.
The integral equality \eqref{N11} represents a weak formulation of the total energy balance.
Note that for $\varphi \in \W^{1,q}_0 (\Omega; \R^2)$, the momentum
equation \eqref{N3} can used to eliminate the integrals containing $\Ov{\vu}$ which gives
\begin{equation}
\label{N12}
\begin{aligned}
&\intO{ \Big( \frac{1}{2} T(\vr) |\vu|^2 + \vr e_R(\vr, \vt)    \Big) \vu \cdot \Grad \varphi }
+ \intO{p_R (\vr, \vt) \vu \cdot \Grad\varphi   } + \intO{ \vQ_\delta \cdot \Grad \varphi }\\&
- \intO{ \mathbb{S}_\delta (\vt, \Grad \vu) \cdot \vu \cdot \Grad \varphi  }
=  - \intO{ \varphi \vr (G + \ep)  } - \intO{\varphi T(\vr) \vc{f} \cdot \vu}
\end{aligned}
\end{equation}
for any $\varphi \in \W^{1,r}_0({\Omega})$. Note that \eqref{N12} can be formally interpreted
as the energy equation, namely we have
\begin{equation}
\label{N13}
\begin{aligned}
&\Div \Bigl[ \Bigl( \frac{1}{2} T(\vr) |\vu|^2 + \vr e_R(\vr, \vt)  \Bigr) \vu \Bigr]
+ \Div (p_R(\vr, \vt) \vu) + \Div \vQ_\delta\\&= \Div( \mathbb{S}_\delta (\vt, \Grad \vu) \cdot \vu)
+ T(\vr) \vc{f} \cdot \vu + \vr (G + \ep).
\end{aligned}
\end{equation}

\subsection{Entropy inequality}

We conclude the limit passage $N \to \infty$ by reporting the entropy inequality:
\begin{equation}
\label{N14}
\begin{aligned}
&\intO{ \varphi
\frac{1}{\vt} \Bigl[ \mathbb{S}_\delta (\vt, \Grad \vu) : \Grad \vu + \frac{ \kappa_\delta (\vt) }{\vt}
|\Grad \vt |^2 \Bigr] } + \intO{ \frac{\vr}{\vt}(G + \ep) \varphi } \\
&\leq  \int_{\partial \Omega} \varphi L \Bigl( 1 - \frac{\Ov{\vt}}{\vt} \Bigr) {\rm dS} -
\intO{ \Bigl( \frac{\vQ_\delta (\vt, \Grad \vt) }{\vt} \Bigr) \cdot \Grad \varphi }
- \intO{ ( \vr  s_R(\vr, \vt)  \vu) \cdot \Grad \varphi } \\
&+ \ep \intO{ \varphi( \Delta \vr - ( \vr - \vr_M))( \vr h_R (\vr) -
s_R(\vr, \vt)) }
\end{aligned}
\end{equation}
for any $\varphi \in \W^{1, \infty}(\Omega)$, $\varphi \geq 0$,
that can be easily deduced from \eqref{A8} as well as the convergence in \eqref{N1}.

\section{Limit $R \to \infty$}
\label{R}

Our next goal is to let the truncation parameter $R \to \infty$ and thus to establish some uniform
estimates on the density.
Let $\{ \vr_R, \vu_R, \vt_R \}_{R > 0}$ be the associated sequence of solutions to the
problem \eqref{N2}, \eqref{N3}, \eqref{N11}, satisfying also the entropy inequality
\eqref{N14}.
For fixed values of the parameters $\ep$ and $\delta$,
the estimates \eqref{A26} and \eqref{A27} remain valid. Then passing to
the subsequences, if necessary, we find
\begin{subequations}
\label{R1}
\begin{align}
\vr_R &\rightharpoonup \vr \ \mbox{weakly in}\ \W^{2,q}(\Omega)\ \mbox{for any}\ 1\leq q < 2,\\
\vt_R &\rightharpoonup \vt \ \mbox{weakly in}\ \W^{1,2}(\Omega),\\
\vt^\beta_R &\rightharpoonup \vt^\beta \ \mbox{weakly in}\ \W^{1,r}(\Omega)
\text{ and }\vt^\beta_R \to \vt^\beta \ \mbox{in}\ \L^s(\Omega)
\ \mbox{for any}\ \beta \in \R,\ 1 \leq r < 2,\ 1 \leq s < \infty, \\
\vu_R &\rightharpoonup \vu \ \mbox{weakly in}\ \W^{1,r}(\Omega; \R^2) \ \mbox{for any}\ 1 \leq r < 2,
\end{align}
\end{subequations}
as $R \to \infty$.
Moreover, we assume that
\begin{equation} \label{R1bis}
h'_R \Bigl( \Ov{\vr} - \frac{1}{R} \Bigr) \to \infty\ \mbox{as}\ R \to \infty.
\end{equation}

\subsection{Pressure estimates}

We derive bounds on the pressure $p_R(\vr_R, \vt_R)$, uniform for $R \to \infty$.
To this end, we introduce an inverse operator
$\mathcal{B}$ to $\Div$ constructed by Bogovskii \cite{BOG}.
The following properties of $\mathcal{B}$ are nowadays rather standard, we refer to the
monograph by Galdi \cite{GAL} for the proofs.

\begin{itemize}
\item The operator $\mathcal{B}$ satisfies
\begin{equation*}
\Div( \mathcal{B} [h]) = h\quad\text{and}\quad \mathcal{B}[h]_{|_{\partial \Omega}} = 0
\ \mbox{for any} \ h \in \L^q(\Omega),\ 1 < q < \infty\quad\text{and}\quad \intO{ h } = 0.
\end{equation*}

\item \(\mathcal{B}\) can be extended to functions \(h\in\L^q(\Omega)\),
\begin{equation}
\label{R2}
\|\mathcal{B}[h]\|_{\W^{1,q}_0 (\Omega; \R^2)} \leq c(q) \| h \|_{\L^q(\Omega)}.
\end{equation}

\item If $h \in \L^q(\Omega)$, $1 < q < \infty$, $h = \Div \vc{g}$, $\vc{g}
\in \L^r(\Omega;\R^2)$ such that
$\vc{g} \cdot \vc{n}_{|_{\partial \Omega}} = 0$, then
\begin{equation}
\label{R3}
\|\mathcal{B}[h]\|_{\L^r (\Omega;\R^2)} \leq c(p,r) \| \vc{g} \|_{\L^r(\Omega;\R^2)}.
\end{equation}

\end{itemize}
The first step is to consider
\begin{equation*}
\bfphi = \mathcal{B} [\vr - \vr_M]
\end{equation*}
as a test function in the approximate momentum equation \eqref{N3}, we get
\begin{equation}
\label{R4}
\begin{aligned}
&\intO{( p_R(\vr_R, \vt_R) - p_R(\vr_M, \vt_R)) ( \vr_R - \vr_M) }
+ \intO{ p_R(\vr_M, \vt_R) \vr_R }  \\
&= \intO{ \Bigl[ \frac{1}{2} T(\vr_R) (\vu_R \cdot \Grad \vu_R \cdot \mathcal{B} [\vr_R - \vr_M]
- \vu_R \otimes \vu_R : \Grad \mathcal{B} [\vr _R- \vr_M] )\Bigr] }
\\ &+ \intO{ \Bigl[ \mathbb{S}_\delta (\vt_R, \Grad \vu_R) : \Grad \mathcal{B} [\vr_R
- \vr_M] - T(\vr_R) \vc{f} \cdot \mathcal{B} [\vr_R - \vr_M] \Bigr] } + \intO{ p_R (\vr_M, \vt_R) \vr_M }.
\end{aligned}
\end{equation}
Observe that
\begin{equation*}
( p_R(\vr_R, \vt_R) - p_R(\vr_M, \vt_R)) ( \vr_R - \vr_M) =
\vt_R( \vr_R h_R(\vr) - \vr_M h_R(\vr_M) )( \vr_R - \vr_M).
\end{equation*}
In virtue of the construction of the truncation $h_R$ specified in \eqref{H3}
and \eqref{R1bis}, we have
\begin{equation*}
\frac{\partial}{\partial \vr} (\vr h_R(\vr)) = h_R(\vr) + \vr h'_R(\vr) \ageq \vr  \ \mbox{for all}\ \vr \geq 0,
\end{equation*}
Notice that
\begin{equation*}
\bigl|\vr_R h_R(\vr_R)-\vr_Mh_R(\vr_M)\bigr|
\ageq
\Bigl|
\int_{\vr_M}^{\vr_R}z\,{\rm d}z
\Bigr|\ageq
\frac12(\vr_R-\vr_M)^2,
\end{equation*}
which implies that
\begin{equation*}
\vt_R | \vr_R - \vr_M |^3 \aleq  ( p_R(\vr_R, \vt_R) - p_R(\vr_M, \vt_R))
( \vr_R - \vr_M).
\end{equation*}
Consequently, we may rewrite inequality \eqref{R4} in the following form
\begin{equation}
\label{R5}
\begin{aligned}
&\frac{1}{2}\intO{( p_R(\vr_R, \vt_R) - p_R(\vr_M, \vt_R))( \vr_R - \vr_M) }
+ \frac{1}{2} \bigl\| \vt_R^{\frac{1}{3}} ( \vr_R - \vr_M) \bigr\|^3_{\L^3(\Omega)}   \\
&\aleq  \intO{ \Bigl[ \frac{1}{2} T(\vr_R) \big(\vu_R \cdot \Grad \vu_R \cdot \mathcal{B}
[\vr_R - \vr_M] - \vu_R \otimes \vu_R : \Grad \mathcal{B} [\vr_R - \vr_M] \big)\Bigr] }
\\ &+ \intO{ \Bigl[ \mathbb{S}_\delta (\vt_R, \Grad \vu_R) :
\Grad \mathcal{B} [\vr_R - \vr_M] - T(\vr_R) \vc{f} \cdot \mathcal{B} [\vr_R - \vr_M] \Bigr] }
+ \intO{ p_R (\vr_M, \vt_R) \vr_M }.
\end{aligned}
\end{equation}
Keeping in mind the uniform bounds in \eqref{R1}, it is easy to control
the integrals on the right--hand side of \eqref{R5} by those on the left--hand side.
Indeed, in view of the properties of the operator $\mathcal{B}$ listed in \eqref{R2} and
\eqref{R3}, we have
\begin{equation*}
\begin{aligned}
&\| \mathcal{B}[\vr_R - \vr_M] \|_{\L^q(\Omega; \R^2)} +
\| \Grad \mathcal{B}[\vr_R - \vr_M] \|_{\L^q(\Omega; \R^4)} \\ &\aleq \| \vr_R - \vr_M \|_{\L^q(\Omega)} =
\bigl\| \vt_R^{-\frac{1}{3}} \vt_R^{\frac{1}{3}} (\vr_R - \vr_M) \bigr\|_{\L^q(\Omega)}
\leq c(\delta, q) \bigl\| \vt_R^{\frac{1}{3}} (\vr_R - \vr_M) \bigr\|_{\L^3(\Omega)}
\ \mbox{for any}\ 1 \leq q < 3.
\end{aligned}
\end{equation*}
Consequently, in view of the uniform bounds \eqref{R1}, we deduce from \eqref{R5}
that the integrals on the left--hand side are bounded uniformly for $R \to \infty$, which
implies that
\begin{equation*}
\frac12\|\vt_R^{1/3}(\vr_R-\vr_M)\|_{\L^3(\Omega)}^3+
\int_{\Omega}p_R(\vr_R,\vt_R)(\vr_R-\vr_M)\ {\rm d}x
\leq c(\delta)+\int_{\Omega}p_R(\vr_M,\vt_R)(\vr_R-\vr_M).
\end{equation*}
From \eqref{R1} we may deduce that \(\intO{p_R(\vr_M,\vr_M)(\vr_R-\vr_M)}\)
is controlled by \(\|\vt_R^{1/3}(\vr_R-\vr_M)\|_{\L^3(\Omega)}^3\). By considering
two cases: \(\vr_R\leq \vr_M+\eta\) and \(\vr_R> \vr_M+\eta\), respectively, with \(\eta>0\)
small enough, we finally get
\begin{equation}
\label{R7}
{\{ p^\gamma_R\}}_m \eqldef\frac{1}{|\Omega|} \intO{ p^\gamma_R(\vr_R, \vt_R) } \leq
c(\delta),\
0 \leq \gamma \leq 1.
\end{equation}
The second step consists to repeat the same procedure with
\begin{equation*}
\mathcal{B} \bigl[ p^\gamma_R - {\{ p^\gamma_R\}}_m \bigr], \ 0< \gamma < 1.
\end{equation*}
Similarly using \eqref{R7}, we deduce that
\begin{equation}
\label{R8}
\begin{aligned}
&\intO{ p_R^{\gamma + 1} }
\leq \intO{ \Bigl[ \frac{1}{2} T(\vr_R) \bigl(\vu_R \cdot \Grad \vu_R \cdot \mathcal{B}
[ p^{\gamma}_R - {\{ p_R^{\gamma}\}}_m] - \vu_R \otimes \vu_R : \Grad \mathcal{B}
[ p^{\gamma}_R - {\{ p^{\gamma}_R \}}_m]  \big)\Bigr] }
\\ &+ \intO{ \Bigl[ \mathbb{S}_\delta (\vt_R, \Grad \vu_R) : \Grad \mathcal{B}
\bigl[ p^{\gamma}_R - {\{ p^{\gamma}_R \}}_m]
- T(\vr) \vc{f} \cdot \mathcal{B}[ p^{\gamma}_R - {\{p_R^{\gamma}\}}_m] \bigr]}.
\end{aligned}
\end{equation}
Once again using the bounds \eqref{R1} combined with the properties of the operator
$\mathcal{B}$, we may infer that all integrals
on the right--hand side of \eqref{R8} can be controlled,
modulo a multiplicative constant, by the following norm
\begin{equation*}
{\|p^\gamma_R\|}_{\L^q(\Omega)} \ \mbox{as soon as}\ q > 2.
\end{equation*}
Thus for $q = \frac{\gamma + 1}{\gamma}$, we may conclude that
\begin{equation*}
\intO{ p_R^{\gamma + 1} } \leq c(\delta) \bigl( 1 + \| p_R \|^\gamma_{\L^{\gamma + 1}(\Omega)}
\bigr),
\end{equation*}
and, consequently, we get
\begin{equation}
\label{R9}
\intO{ p_R^{\gamma + 1}(\vr_R, \vt_R) } \leq c(\delta, \gamma),\ 0
< \gamma < 1 \ \mbox{uniformly for}\ R \to \infty.
\end{equation}
Finally, writing
\begin{equation*}
\vr_R h_R(\vr_R) = \vt_R^{-1}
p_R(\vr_R, \vt_R),
\end{equation*}
we also obtain
\begin{equation}
\label{R10}
\intO{ | \vr_R h_R(\vr_R) |^\omega } \leq c(\delta, \omega),\
0 < \omega < 2 \ \mbox{uniformly for}\ R \to \infty.
\end{equation}

\subsection{Convergence and the limit system}

With \eqref{R1}, \eqref{R9} and \eqref{R10} at hand, it is standard to
perform the limit for $R \to \infty$ in the system of approximate equations.
Moreover, as the limit pressure is singular at $\Ov{\vr}$ (see hypothesis \eqref{M2}), we deduce
from \eqref{R10} that
\begin{equation} \label{R11}
0 \leq \vr < \Ov{\vr} \ \mbox{a.e in}\ \Omega\quad\text{and}\quad
\| p(\vr, \vt) \|_{\L^\omega(\Omega)} \leq c(\delta, \omega) \ \mbox{for any}\ 1\leq \omega < 2,
\end{equation}
cf. also \cite{CiFeJaPe}.
Accordingly, the limit system of equations reads as follows:
\begin{subequations}
\label{R12-13}
\begin{align}
\label{R12}
&\Div (\vr \vu) = \ep \Delta \vr - \ep (\vr - \vr_M),\quad
\Grad \vr \cdot \vc{n}_{|_{\partial \Omega}} = 0,\quad
0 \leq \vr < \Ov{\vr} , \quad \frac{1}{|\Omega|} \intO{ \vr } = \vr_M,\\&
\label{R13}
\intO{ \Bigl[ \frac{1}{2} \vr( \vu \otimes \vu : \Grad \bfphi -
\vu \cdot \Grad \vu \cdot \bfphi) + p (\vr, \vt)  \Div \bfphi \Bigr] }
= \intO{ \Bigl[ \mathbb{S}_\delta (\vt, \Grad \vu) : \Grad \bfphi - \vr \vc{f} \cdot \bfphi \Bigr] }
\end{align}
\end{subequations}
for any $\bfphi \in \W^{1,q}_0(\Omega; \R^2)$, $q > 2$. Notice that $(\vu - \vuB) \in \W^{1,r}_0
(\Omega; \R^2)$, $1 \leq r < 2$ and
\begin{equation}
\label{R14}
\begin{aligned}
&\intO{ \Big( \frac{1}{2} \vr |\vu|^2 + \vr e(\vr, \vt)    \Big) \vu \cdot \Grad \varphi }
+ \intO{p (\vr, \vt) \vu \cdot \Grad\varphi   } + \intO{ \vQ_\delta \cdot \Grad \varphi }
\\&- \intO{ \mathbb{S}_\delta (\vt, \Grad \vu) \cdot \vu \cdot \Grad \varphi  }
= \int_{\partial \Omega} L (\vt - \Ov{\vt})
\varphi \ {\rm dS} - \intO{ \varphi \vr (G + \ep)  } - \intO{\varphi \vr \vc{f} \cdot \vu}
\\
&+ \intO{ \frac{1}{2} \vr \big(
 (\vu \otimes \vu) : \Grad (\varphi \Ov{\vu}) - \vu \cdot \Grad \vu \cdot (\varphi \Ov{\vu})  \big) }
+ \intO{ p (\vr, \vt) \Div (\varphi \Ov{\vu}) } \\
 &+ \intO{\vr \vc{f} \cdot
(\varphi \Ov{\vu}) } - \intO{ \mathbb{S}_\delta (\vt, \Grad \vu) : \Grad (\varphi \Ov{\vu}) }
 \end{aligned}
\end{equation}
for any $\varphi \in \W^{1,\infty}({\Omega})$;
together with the entropy inequality
\begin{equation}
\label{R15}
\begin{aligned}
&\intO{ \varphi
\frac{1}{\vt} \Bigl[ \mathbb{S}_\delta (\vt, \Grad \vu) : \Grad \vu + \frac{ \kappa_\delta (\vt) }{\vt}
|\Grad \vt |^2 \Bigr] } + \intO{ \frac{\vr}{\vt}(G + \ep) \varphi } \\
&\leq  \int_{\partial \Omega} \varphi L \Bigl( 1 - \frac{\Ov{\vt}}{\vt} \Bigr) {\rm dS}-
\intO{\Bigl( \frac{\vQ_\delta (\vt, \Grad \vt) }{\vt}\Bigr) \cdot \Grad \varphi }
- \intO{( \vr  s(\vr, \vt)  \vu) \cdot \Grad \varphi } \\
&+ \ep \intO{ \varphi( \Delta \vr - ( \vr - \vr_M))( \vr h (\vr) -
s(\vr,\vt)) }
\end{aligned}
\end{equation}
for any $\varphi \in \W^{1, \infty}(\Omega)$, $\varphi \geq 0$.

\section{Limit $\ep \to 0$}
\label{E}

The process $\ep \to 0$ is crucial as it requires strong convergence of the approximate densities.
We use the approach proposed by Lions in \cite{LI4}, based on the monotonicity
of the pressure, combined with the Commutator Lemma, introduced in \cite{EF71},
to handle the temperature fluctuations of the viscosity coefficients.
Keeping $\delta > 0$ fixed, we consider a family $\{ \vre, \vue, \vte \}_{\ep > 0}$
of solutions of the approximate system
(\ref{R12}--\ref{R15}). Given the available $\delta-$dependent estimates derived
in the preceding part, passing to the subsequences, if necessary, we find
\begin{subequations}
\label{E1}
\begin{align}
\vre &\rightharpoonup\vr \ \mbox{weakly-$\star$ in}\ \L^\infty(\Omega)\ \mbox{with}\ 
0 \leq \vr \leq \Ov{\vr},\\
\vte &\rightharpoonup \vt \ \mbox{weakly in}\ \W^{1,2}(\Omega),\\
\vte^\beta &\rightharpoonup \vt^\beta \ \mbox{weakly in}\ \W^{1,r}(\Omega)
\ \mbox{and}\ \vte^\beta \to \vt^\beta \ \mbox{in}\ \L^s(\Omega)
\ \mbox{for any}\ \beta \in \R,\ 1 \leq r < 2,\ 1 \leq s < \infty, \\
\vue &\rightharpoonup \vu \ \mbox{weakly in}\ \W^{1,r}(\Omega; \R^2) \ \mbox{for any}\ 1 \leq r < 2,
\end{align}
\end{subequations}
as $\ep \to 0$.
Moreover, as a consequence
of \eqref{R11}, we have
\begin{equation}
\label{E2a}
p_\ep = p(\vre, \vte) \to \Oov{p(\vr, \vt)} \ \mbox{weakly in}\ \L^\omega(\Omega)\ \mbox{for any}\
1 \leq \omega < 2.
\end{equation}

\subsection{Strong convergence of approximate densities}

Our goal is to show, up to a suitable subsequence,
\begin{equation} \label{E2}
\vre \to \vr \ \mbox{a.e in}\ \Omega.
\end{equation}
The proof is based on monotonicity of the pressure in the density variable, cf. hypothesis \eqref{M2}.
Similarly to \cite{CiFeJaPe}, we show that
\begin{equation} \label{E3}
\Oov{ p(\vr, \vt) \vr } = \Oov{p(\vr, \vt)} \vr \ \mbox{a.e in}\ \Omega,
\end{equation}
where the bar is used to denote a weak limit of the corresponding composition.
In view of the strong convergence of the
temperature in \eqref{E1}, relation \eqref{E3} gives rise to
\begin{equation*}
\vt \ \Oov{ \vr h(\vr) \vr } = \vt \ \Oov{h(\vr) \vr} \ \vr,
\end{equation*}
but since $\vt > 0$ almost everywhere in $\Omega$, this yields
\begin{equation*}
\Oov{ \vr h(\vr) \vr } = \Oov{h(\vr) \vr} \ \vr \ \mbox{a.e in}\ \Omega.
\end{equation*}
The function $\vr \mapsto \vr h(\vr)$ being (strictly) increasing, cf. \eqref{M2},
this implies \eqref{E2}, exactly as in
\cite{CiFeJaPe}.

Following the approach of Lions \cite{LI4}, we derive \eqref{E3}
from the effective viscous flux identity. To this end, we first
perform the limit in the momentum equation \eqref{R13}:
\begin{equation}
\label{E4}
\intO{ \Bigl[ \frac{1}{2}  \Big( \Oov{\vr \vu \otimes \vu} : \Grad \bfphi -
\Oov{\vr \vu \cdot \Grad \vu } \cdot \bfphi \Big) + \Oov{p (\vr, \vt)}  \Div \bfphi \Bigr] }
 = \intO{ \Bigl[ \mathbb{S}_\delta (\vt, \Grad \vu) : \Grad \bfphi - \vr \vc{f} \cdot \bfphi \Bigr] }
\end{equation}
for any $\bfphi \in \W^{1,q}_0(\Omega; \R^2)$, $q > 2$, $(\vu - \vuB) \in
\W^{1,r}_0(\Omega; \R^2)$, $1 \leq r < 2$. Note that
$\Oov{\mathbb{S}_\delta (\vt, \Grad \vu)} = \mathbb{S}_\delta (\vt, \Grad \vu)$
thanks to the strong convergence of the approximate temperatures.

Now, we repeat the same process with the test function
\begin{equation*}
\bfphi = \phi \Grad \Delta^{-1} [\phi \vre] \ \mbox{where}\ \phi \in \DC(\Omega),\
0 \leq \phi \leq 1,
\end{equation*}
and $\Delta^{-1}$ is the inverse of the Laplace operator defined by means of the
Green function on $\R^2$. Plugging $\bfphi$ in
\eqref{R13}, peforming the limit and regrouping terms in the limit expression, we find
\begin{equation}
\label{E5}
\begin{aligned}
&\intO{ \Bigl[ \frac{1}{2}  \bigl( \Oov{\vr \vu \otimes \vu : \Grad \bigl( \phi \Grad \Delta^{-1}
[\phi \vr] \bigr) }-
\phi \Oov{\vr \vu \cdot \Grad \vu } \cdot \Grad \Delta^{-1} [\phi \vr] \bigr) +
\phi \vr \vc{f} \cdot \Grad \Delta^{-1} [\phi \vr] \Bigr] }
\\& = \intO{ \Bigl[ \Oov{ \mathbb{S}_\delta (\vt, \Grad \vu) : \Grad \bigl(
\phi \Grad \Delta^{-1} [\phi \vr ] \bigr)} - \Oov{p (\vr, \vt)  \Div \bigl( \phi \Grad
\Delta^{-1} [\phi \vr]  \bigr)} \Bigr] }.
\end{aligned}
\end{equation}
Note that, thanks to the regularizing properties of the operator $\Delta^{-1}$,  we have
\begin{equation}
\label{E5bis}
\phi \Grad \Delta^{-1} [\phi \vre] \to \phi \Grad \Delta^{-1} [\phi \vr]
\ \mbox{(strongly) in}\ \C^0(\Ov{\Omega}).
\end{equation}
Finally, we use the quantity
\begin{equation*}
\bfphi = \phi \Grad \Delta^{-1} [\phi \vr]
\end{equation*}
as a test function in the limit equation \eqref{E4}, we get
\begin{equation}
\label{E6}
\begin{aligned}
&\intO{ \Bigl[ \frac{1}{2}  \bigl( \Oov{\vr \vu \otimes \vu} : \Grad \bigl( \phi \Grad
\Delta^{-1} [\phi \vr] \bigr) -
\phi \Oov{\vr \vu \cdot \Grad \vu } \cdot \Grad \Delta^{-1} [\phi \vr] \bigr)
+ \phi \vr \vc{f} \cdot \Grad \Delta^{-1} [\phi \vr] \Bigr] }
\\& = \intO{ \Bigl[  \mathbb{S}_\delta (\vt, \Grad \vu) : \Grad \bigl(
\phi \Grad \Delta^{-1} [\phi \vr ] \bigr) - \Oov{p (\vr, \vt)}  \Div \bigl( \phi \Grad \Delta^{-1} [\phi \vr]
\bigr) \Bigr] }.
\end{aligned}
\end{equation}
Now, we compare the terms on the right--hand sides of \eqref{E5}, \eqref{E6}.
As the velocity converges strongly, we have
\begin{equation*}
\begin{aligned}
\Oov{\vr \vu \otimes \vu : \Grad \bigl( \phi \Grad \Delta^{-1} [\phi \vr] \bigr) }
&= \vu \cdot \Oov{\vr \vu \cdot \Grad \bigl( \phi \Grad \Delta^{-1} [\phi \vr] \bigr) }\\
\Oov{\vr \vu \otimes \vu} : \Grad \bigl( \phi \Grad \Delta^{-1} [\phi \vr] \bigr)
&= \vu \cdot \Oov{\vr \vu} \cdot \Grad \bigl( \phi \Grad \Delta^{-1} [\phi \vr] \bigr) .
\end{aligned}
\end{equation*}
Next, we observe, exactly as in \cite{CiFeJaPe} that
\begin{equation}
\label{E6bis}
\Oov{\vr \vu \cdot \Grad \bigl( \phi \Grad \Delta^{-1} [\phi \vr] \bigr) } =
\Oov{\vr \vu} \cdot \Grad \bigl( \phi \Grad \Delta^{-1} [\phi \vr] \bigr).
\end{equation}
To this end, we use Div--Curl Lemma (see Tartar \cite{Tar79CCPD}), we get
\begin{equation*}
{\rm curl}\Grad \bigl( \phi \Grad \Delta^{-1} [\phi \vre] \bigr)  = 0,
\end{equation*}
and
\begin{equation*}
\Div (\vre \vue) = \ep \Delta \vre - \ep (\vre - \vr_M).
\end{equation*}
Now, we take $\varphi = 1$ in the entropy inequality \eqref{R15}, we find
\begin{equation}
\label{E8}
\begin{aligned}
&\intO{
\frac{1}{\vt} \Bigl[ \mathbb{S}_\delta (\vte, \Grad \vue) : \Grad \vue + \frac{ \kappa_\delta (\vte) }{\vte}
|\Grad \vte |^2 \Bigr] } + \intO{ \frac{\vre}{\vt}(G + \ep) }\\&
+ \ep \intO{ \Bigl( h(\vre) + \vre h'(\vre) + \frac{h(\vre)}{\vre} \Bigr) |\Grad \vre |^2 }\\
&+ \ep \intO{ (\vre - \vr_M) \Bigl( \vre h(\vre) - \vr_M h(\vr_M) + \int_{\vr_M}^{\vre}
\frac{h(z)}{z} \ {\rm d}z \Bigr) }
\\
&\leq  \int_{\partial \Omega} L \Bigl( 1 - \frac{\Ov{\vt}}{\vte} \Bigr) {\rm dS} +
- \ep c_v \intO{ \Grad \vre \cdot \Grad \log(\vte) }
+\ep \intO{ ( \vre - \vr_M)
c_v \log(\vte) }.
\end{aligned}
\end{equation}
In particular, we deduce
\begin{equation}
\label{E9}
\ep \| \Grad \vre \|^2_{\L^2(\Omega)} \leq c(\delta).
\end{equation}
We may deduce from \eqref{A1} and \eqref{E9}
that \(\Div(\vr_{\varepsilon}\vu_{\varepsilon})\rightarrow 0\)
in \(\W^{-1,2}(\Omega)\) then
\(\Div (\vre \vue)\) belongs to a compact set in \(\W^{-1,2}(\Omega)\).
Thus relation \eqref{E6bis} follows directly from Div--Curl Lemma.

Comparing \eqref{E5} and \eqref{E6}, we obtain
\begin{equation*}
\begin{aligned}
&\intO{ \Bigl[ \Oov{ \mathbb{S}_\delta (\vt, \Grad \vu) : \Grad \bigl(
\phi \Grad \Delta^{-1} [\phi \vr ] \bigr)} - \Oov{p (\vr, \vt)  \Div \bigl( \phi \Grad \Delta^{-1}
[\phi \vr]    \bigr)} \Bigr] }\\
&=\intO{ \Bigl[  \mathbb{S}_\delta (\vt, \Grad \vu) : \Grad \bigl(
\phi \Grad \Delta^{-1} [\phi \vr ] \bigr) - \Oov{p (\vr, \vt)}  \Div \bigl( \phi \Grad \Delta^{-1} [\phi \vr]
\bigr) \Bigr] }
\end{aligned}
\end{equation*}
that can be simplified via \eqref{E5bis} to
\begin{equation}
\label{E12}
\begin{aligned}
\intO{ \Bigl[ \phi^2 \Oov{ p(\vr, \vt) \vr } - \Oov{p(\vr, \vt)} \vr \Bigr] }&=
\intO{ \phi \bigl(  \Oov{ \mathbb{S}_\delta (\vt, \Grad \vu) :
\Grad \Delta^{-1} \Grad [\phi \vr ] } - \mathbb{S}_\delta (\vt, \Grad \vu) : \Grad
\Delta^{-1} \Grad [\phi \vr ] \bigr)  }\\
&=
\intO{ \phi \bigl(  \Oov{ \Grad \Delta^{-1} \Grad : (\phi \mathbb{S}_\delta (\vt, \Grad \vu))
\vr  } - \Grad
\Delta^{-1} \Grad  : ( \phi \mathbb {S}_\delta (\vt, \Grad \vu)  \vr  )\bigr)  }
\end{aligned}
\end{equation}
Our plan consists in replacing
\(\Grad \Delta^{-1} \Grad : ( \phi \mathbb{S}_\delta (\vt, \Grad \vu))\) by
\((\mu_{\delta} (\vt) + g(\vt) ) \Div \vu\)
in the identity \eqref{E12} where \(g\) is a polynomial increasing function.
To this end, write
\begin{equation*}
\Grad \Delta^{-1} \Grad : ( \phi \mathbb{S}_\delta (\vt, \Grad \vu)) =
\bigl( \Grad \Delta^{-1} \Grad : ( \phi \mathbb{S}_\delta (\vt, \Grad \vu)) -
(\mu_{\delta} (\vt) + g (\vt) ) \Div \vu \bigr)+ (\mu_{\delta} (\vt) + g(\vt) ) \Div \vu.
\end{equation*}
The expression in the curly brackets is a commutator of the pseudo--differential
operator $\Grad \Delta^{-1} \Grad$
with multiplication by a function of $\vt$. It enjoys extra compactness properties
exploited in \cite{EF71}. We report the following result that can be see as a version
of the abstract results of Coifman and Meyer \cite{COME}:

\begin{Lemma}[Commutator Lemma]
\label{EL1}
Let $w \in \W^{1,r}(\R^N)$ and $\vc{V} \in \L^q(\R^N; \R^N)$ be given fields,
\begin{equation*}
1 < r < N, \quad 1 < q < \infty,\quad \frac{1}{r} + \frac{1}{q} < 1 + \frac{1}{N}.
\end{equation*}
Then for any $s$ satisfying
\begin{equation*}
\frac{1}{r} + \frac{1}{q} - \frac{1}{N} < \frac{1}{s} < 1,
\end{equation*}
there exists $\beta \in (0,1)$ such that
\begin{equation*}
\| \Grad \Delta^{-1} \Grad \cdot [w \vc{V}] - w \Grad \Delta^{-1} \Grad \cdot [\vc{V}]
\|_{\W^{\beta,s}(\R^N; \R^N)}
\aleq \| w \|_{\W^{1,r}(\R^N)} \| \vc{V} \|_{\L^q(\R^N; \R^N)}.
\end{equation*}
\end{Lemma}
We apply Lemma \ref{EL1} to
\begin{equation*}
w = \phi \mu_\delta (\vte), \quad \vc{V} = \Grad u^i_\ep, \quad i = 1,2,
\quad N = 2,\quad r < 2,\quad q < 2,
\end{equation*}
and we deduce
the strong convergence of the commutator in \(\L^2\)-norm.
Accordingly, we may deduce from \eqref{E12} the desired relation:
\begin{equation}
\label{E13}
\begin{aligned}
\intO{ \Bigl[ \phi^2 \Oov{ p(\vr, \vt) \vr } - \Oov{p(\vr, \vt)} \vr \Bigr] } =
\intO{ \phi^2 ( \mu_{\delta}(\vt) + g(\vt))
( \Oov{ \vr \Div \vu } - \vr \Div \vu)}.
\end{aligned}
\end{equation}
Relation \eqref{E13} is called Lions' identity. One can deduce \eqref{E3},
and, consequently, the strong convergence of the approximate densities from \eqref{E13}.
The details of this procedure are detailed in \cite{CiFeJaPe}.

\subsection{Convergence and the limit system}

Once strong convergence of the densities has been established, it is straightforward
to pass to the limit in the approximate equations.
Note that $\sqrt{\ep} \Grad \vre$ is bounded in the $\L^2$-norm uniformly
for $\ep \to 0$. Consequently, letting $\ep \to 0$ in
\eqref{R12-13}--\eqref{R15}, we obtain
\begin{subequations}
\label{E14-15}
\begin{align}
\label{E14}
&\intO{ \vr \vu \cdot \Grad \varphi }  = 0 \ \mbox{for any}\ \varphi \in \W^{1, \infty}(\Omega),\quad
0 \leq \vr < \Ov{\vr} , \quad \frac{1}{|\Omega|} \intO{ \vr} = \vr_M,\\&
\label{E15}
\intO{ \Bigl[ \frac{1}{2} \vr( \vu \otimes \vu : \Grad \bfphi -
\vu \cdot \Grad \vu \cdot \bfphi) + p (\vr, \vt)  \Div \bfphi \Bigr] }
= \intO{ \Bigl[ \mathbb{S}_\delta (\vt, \Grad \vu) : \Grad \bfphi - \vr \vc{f} \cdot \bfphi \Bigr] }
\end{align}
\end{subequations}
for any $\bfphi \in \W^{1,q}_0(\Omega; \R^2)$, $q > 2$.
Since $(\vu - \vuB) \in \W^{1,r}_0(\Omega; \R^2)$,
$1 \leq r < 2$, we have
\begin{equation}
\label{E16}
\begin{aligned}
&\intO{ \Big( \frac{1}{2} \vr |\vu|^2 + \vr e(\vr, \vt)    \Big) \vu \cdot \Grad \varphi }
+ \intO{p (\vr, \vt) \vu \cdot \Grad\varphi   } + \intO{ \vQ_\delta \cdot \Grad \varphi }
\\&- \intO{ \mathbb{S}_\delta (\vt, \Grad \vu) \cdot \vu \cdot \Grad \varphi  }
= \int_{\partial \Omega} L (\vt - \Ov{\vt})
\varphi \ {\rm dS}  - \intO{ \varphi \vr G   } - \intO{\varphi \vr \vc{f} \cdot \vu}
\\
&+ \intO{ \frac{1}{2} \vr \big(
 (\vu \otimes \vu) : \Grad (\varphi \Ov{\vu}) - \vu \cdot \Grad \vu \cdot (\varphi \Ov{\vu})  \big) }
+ \intO{ p (\vr, \vt) \Div (\varphi \Ov{\vu}) } \\
 &+ \intO{\vr \vc{f} \cdot
(\varphi \Ov{\vu}) } - \intO{ \mathbb{S}_\delta (\vt, \Grad \vu) : \Grad (\varphi \Ov{\vu}) }
 \end{aligned}
\end{equation}
for any $\varphi \in \W^{1,\infty}({\Omega})$;
and  the entropy inequality
\begin{equation}
\label{E17}
\begin{aligned}
&\intO{ \varphi
\frac{1}{\vt} \Bigl[ \mathbb{S}_\delta (\vt, \Grad \vu) : \Grad \vu + \frac{ \kappa_\delta (\vt) }{\vt}
|\Grad \vt |^2 \Bigr] } + \intO{ \frac{\vr}{\vt} G  \varphi } \\
&\leq  \int_{\partial \Omega} \varphi L \Bigl( 1 - \frac{\Ov{\vt}}{\vt} \Bigr) {\rm dS} -
\intO{ \Bigl( \frac{\vQ_\delta (\vt, \Grad \vt) }{\vt} \Bigr) \cdot \Grad \varphi }
- \intO{( \vr  s(\vr, \vt)  \vu) \cdot \Grad \varphi}.
\end{aligned}
\end{equation}
for any $\varphi \in \W^{1, \infty}(\Omega)$, $\varphi \geq 0$.

\section{Limit $\delta \to 0$}
\label{D}

Our ultimate goal is to perform the limit $\delta \to 0$ recovering the weak
formulation of the original problem. This can be done
in a similar way as in the preceding section, however, we must establish the necessary
uniform bound \emph{independent} of $\delta$.
As the bounds based on the entropy inequality \eqref{A13} hold uniformly
for $\delta \to 0$, we must only establish the bounds on
the temperature similar to those obtained in Section \ref{TM}.

\subsection{Uniform bounds}

Let $\{ \vr_\delta, \vu_\delta, \vt_\delta \}_{\delta > 0}$ be a sequence
of approximate solutions solving \eqref{E14-15}--\eqref{E17}.
Taking $\varphi = 1$ as a test function in the total energy balance \eqref{E16},
similarly to Section \ref{TM}, we obtain
\begin{equation*}
\begin{aligned}
&\int_{ \partial \Omega } \vt_\delta \ {\rm dS} \leq c_{D} +
\intO{ \vr_\delta \vc{f} \cdot \vu_\delta }
- \intO{ \frac{1}{2} \vrd \big(
 (\vud \otimes \vud) : \Grad \Ov{\vu} - \vud \cdot \Grad \vud \cdot \Ov{\vu}  \big) }\\&
 - \intO{\vrd \vc{f} \cdot
 \Ov{\vu} } + \intO{ \mathbb{S}_\delta (\vtd, \Grad \vud) : \Grad \Ov{\vu} }.
 \end{aligned}
\end{equation*}
Moreover, the equation of continuity \eqref{E14} can be used to rewrite the convective
term, we get
\begin{equation}
\label{D1}
\int_{ \partial \Omega } \vt_\delta \ {\rm dS} \leq c_D +
\intO{ \vr_\delta \vc{f} \cdot (\vu_\delta - \Ov{\vu} ) } + \intO{  \vrd \vud \cdot \Grad \vud \cdot \Ov{\vu}   }
 +  \intO{ \mathbb{S}_\delta (\vtd, \Grad \vud) : \Grad \Ov{\vu} }.
\end{equation}
Next, taking $\varphi = 1$ in the entropy inequality \eqref{E17}, we find
\begin{equation}
\label{D2}
\intO{
\frac{1}{\vtd} \Bigl[ \mathbb{S}_\delta (\vtd, \Grad \vud) : \Grad \vud + \frac{ \kappa_\delta (\vtd) }{\vtd}
|\Grad \vtd |^2 \Bigr] } + \intO{ \frac{\vrd}{\vtd} G   } \leq  \int_{\partial \Omega}
L \Bigl( 1 - \frac{\Ov{\vt}}{\vtd} \Bigr) {\rm dS}.
\end{equation}
Our goal, similarly to Section \ref{TM}, is to control all integrals on the
right--hand side of \eqref{D1} by means of a suitable norm
of $\vtd$. First observe that, by virtue of \eqref{korn} and \eqref{porn},  we obtain
\begin{equation}
\label{D3}
\begin{aligned}
\Bigl| \intO{ \vr_\delta \vc{f} \cdot (\vu_\delta - \Ov{\vu} ) } \Bigr|
&\aleq \| \mathbb{D} \vud - \mathbb{D} \Ov{\vu} \|_{\L^r(\Omega; \R^4)}
\leq \| \mathbb{D} \vud \|_{\L^r(\Omega; \R^4)}    +
\| \mathbb{D} \Ov{\vu}\|_{\L^r(\Omega; \R^4)}\\
&= \bigl\| \vtd^{\frac{1}{2} }\vtd^{- \frac{1}{2}} \mathbb{D} \vud \bigr\|_{\L^r(\Omega; \R^4)}
+ \| \mathbb{D} \Ov{\vu}\|_{\L^r(\Omega; \R^4)}
\aleq \| \vtd \|^{\frac{1}{2}}_{\L^s(\Omega)} + \| \mathbb{D} \Ov{\vu}\|_{\L^r(\Omega; \R^4)}
\end{aligned}
\end{equation}
for some $1 < r < 2$, $s \geq 1$. Note that according to the entropy estimates \eqref{D2},
the norm $\vtd^{-\frac{1}{2}} \mathbb{D} \vud$
is bounded in the $\L^2$-norm.
Next, we handle the integral
\begin{equation*}
\Bigl| \intO{ \mathbb{S}_\delta (\vtd, \Grad \vud) : \Grad \Ov{\vu} } \Bigr|
\aleq \| \Grad \Ov{\vu} \|_{\L^q(\Omega; \R^4)}\| (1 + \delta \vtd^a )
| \mathbb{D} \vud | \|_{\L^r(\Omega)},
\ \frac{1}{q} + \frac{1}{r} = 1,
\end{equation*}
where we focus on the case $\alpha = 0$ in \eqref{I10} as otherwise
the estimates would be the same as in Section \ref{TM}.
In view of the entropy estimates \eqref{D2}, we have
\begin{equation}
\label{D4}
\bigl\| \vtd^{-\frac{1}{2}} \mathbb{D} \vud \bigr\|_{\L^2(\Omega; \R^4)} +
\bigl\| \sqrt{\delta} \vtd^{\frac{a - 1}{2}} \mathbb{D} \vud \bigr\|_{\L^2(\Omega; \R^4)}
\leq c_{D}.
\end{equation}
Consequently, by interpolation,  we get
\begin{equation*}
\begin{aligned}
& \bigl\| (1 + \delta \vtd^a ) | \mathbb{D} \vud | \bigr\|_{\L^r(\Omega)} \aleq
\| \mathbb{D} \vud \|_{\L^r(\Omega; \R^4)} +\| \delta \vtd^a \mathbb{D} \vud\|_{\L^r(\Omega;\R^4)}\\
&= \| \vtd^{\frac{1}{2}} \vtd^{-\frac{1}{2}} \mathbb{D} \vud \|_{\L^r(\Omega; \R^4)} +
\sqrt{\delta} \bigl\| \vtd^{\frac{a + 1}{2}} \sqrt{\delta} \vtd^{\frac{a - 1}{2}}
\mathbb{D} \vud \bigr\|_{\L^r(\Omega; \R^4)}
\aleq \| \vtd \|^{\frac{1}{2}}_{\L^s(\Omega)} +
\sqrt{\delta} \| \vtd \|^{\frac{a + 1}{2}}_{\L^s(\Omega)}
\end{aligned}
\end{equation*}
for some $s \geq 1$ as soon as $1 \leq r < 2$.
Thus we may infer that
\begin{equation}
\label{D5}
\Bigl| \intO{ \mathbb{S}_\delta (\vtd, \Grad \vud) : \Grad \Ov{\vu} } \Bigr|
\aleq \| \Grad \Ov{\vu} \|_{\L^q(\Omega; \R^4)} \bigl( \| \vtd \|^{\frac{1}{2}}_{\L^s(\Omega)} +
\sqrt{\delta} \| \vtd \|^{\frac{a + 1}{2}}_{\L^s(\Omega)} \bigr)
\ \mbox{for some}\ s\geq 1 \ \mbox{if}\ q > 2.
\end{equation}
Finally, we have to estimate the integral
\begin{equation*}
\Bigl| \intO{  \vrd \vud \cdot \Grad \vud \cdot \Ov{\vu}   } \Bigr|
\aleq \| \Ov{\vu} \|_{\L^q(\Omega; \R^2)} \| \vud \cdot \Grad \vud \|_{\L^r(\Omega; \R^2)},
\ \frac{1}{q} + \frac{1}{r} = 1.
\end{equation*}
Furthermore, we may notice that
\begin{equation*}
\| \vud \cdot \Grad \vud\|_{\L^r(\Omega; \R^2)} \aleq
\bigl( \| (\vud - \Ov{\vu})  \cdot \Grad (\vud - \Ov{\vu})\|_{\L^r(\Omega; \R^2)} +
\| \Ov{\vu} \|^2_{\W^{1, 2r}(\Omega; \R^2)}+c(\delta) \bigr).
\end{equation*}
Next, we use \eqref{D4} and proceed exactly as in Section \ref{TM} to conclude
\begin{equation}
\label{D6}
\Bigl| \intO{  \vrd \vud \cdot \Grad \vud \cdot \Ov{\vu}   } \Bigr|  \aleq \| \Ov{\vu} \|_{\L^q(\Omega; \R^2)}
\Bigl( \| \vtd \|_{\L^s(\Omega)} + \| \Ov{\vu} \|^2_{\W^{1, q}(\Omega; \R^2)} \Bigr)
\ \mbox{for some} \ q, s \geq 1.
\end{equation}
Summing up \eqref{D1}, \eqref{D3}, \eqref{D5} and \eqref{D6} to get
\begin{equation}
\label{D7}
\begin{aligned}
&\int_{\partial \Omega} \vtd \ {\rm dS} \aleq
1 + \| \vtd \|^{\frac{1}{2}}_{\L^s(\Omega)} + \| \Ov{\vu}\|_{\W^{1,q}(\Omega; \R^4)}^4\\&
+ \| \Ov{\vu} \|_{\W^{1,q}(\Omega; \R^4)} \bigl( \| \vtd \|^{\frac{1}{2}}_{\L^s(\Omega)} +
\sqrt{\delta} \| \vtd \|^{\frac{a + 1}{2}}_{\L^s(\Omega)} \bigr) +
\| \Ov{\vu} \|_{\L^q(\Omega; \R^2)}
\| \vtd \|_{\L^s(\Omega)}
\end{aligned}
\end{equation}
for some finite $s,q \geq 1$.

At this stage, we need the following extension lemma proved in \cite[Lemma A1]{CiFeJaPe}.

\begin{Lemma}
\label{LD1}
Let $\Omega \subset \R^2$  be a bounded Lipschitz domain.
Let $\Ov{\vu} \in \W^{1,p}(\Omega; \R^2)$, $1 < p < \infty$, be given such that
$\Ov{\vu} \cdot \vc{n}_{|_{\partial \Omega}} = 0$. Let $q$ be given such that
\(
1< q < \frac{2p}{2 - p},
\)
if \(p < 2\) and $q > 1$ arbitrary finite otherwise.
Then for any $\omega > 0$, there exists
$\Ov{\vu}_\omega \in \W^{1,p}(\Omega; \R^N)$ with the following properties:
\begin{itemize}

\item \(\Ov{\vu}_\omega = \Ov{\vu} \ \mbox{on}\ \partial \Omega \ \mbox{in the sense of traces,}\)

\item \(\Div \Ov{\vu}_\omega = 0\ \mbox{in}\ \Omega\),

\item \(\|\Ov{\vu}_\omega \|_{\L^q(\Omega; \R^2)} < \omega\),

\item  \(\| \Ov{\vu}_\omega \|_{\W^{1,p}(\Omega; \R^2)}
\leq c(\omega, p,q) \| \Ov{\vu} \|_{\W^{1,p}(\Omega;\R^2)}.\)

\end{itemize}

\end{Lemma}

The idea is to replace $\Ov{\vu}$ by $\Ov{\vu}_\omega$ in the energy balance \eqref{E16},
and, subsequently in \eqref{D7},
to make the coefficient  $\| \Ov{\vu}_\omega \|_{\L^q(\Omega;\R^2)}$
multiplying the highest power of the norm of $\vtd$ small enough.
Then the uniform bound on $\vtd$ is obtained from \eqref{D2} and \eqref{D7}
via a compactness argument. To carry out this program, some preliminaries are necessary.
The first may be seen as a direct consequence of the Sobolev embedding
$\W^{1,2} \hookrightarrow \L_\Phi$ already used in Section \ref{TM}.

\begin{Lemma}
\label{DL2}
Let $\Omega \subset \R^2$ be a bounded Lipschitz domain.
There exists a function
\begin{equation*}
\chi\eqldef \chi(\Lambda_1, \Lambda_2,s): [0, \infty)^2 \times [1, \infty) \to \R
\end{equation*}
with the following property: If $\psi > 0$ a.e in $\Omega$ and there exist \(\Lambda_1,
\Lambda_2\geq 0\) such that
\begin{equation*}
\| \Grad \log (\psi) \|_{\L^2(\Omega)} \leq \Lambda_1\quad\text{and}\quad
\int_{\partial \Omega} \psi \ {\rm dS} \leq \Lambda_2,
\end{equation*}
then
\begin{equation*}
\| \psi \|_{\L^s(\Omega)} \leq \chi (\Lambda_1, \Lambda_2, s).
\end{equation*}
\end{Lemma}
Next, we show the following:
\begin{Lemma}
\label{DL3}
Let $\Omega \subset \R^2$ be a bounded Lipschitz domain.
Let $\Lambda_1 \geq 0$, $Z\geq 0$, $s \geq 1$, $\beta \in (0,1)$, and $\omega > 0$ be such that
\begin{equation*}
\omega \chi(\Lambda_1,1, s) < 1,
\end{equation*}
where $\chi$ is the function identified in Lemma \ref{DL2}.
Then there exists
\(C\eqldef C(\Lambda_1, Z,s, \beta, \omega)\)
such that
\begin{equation*}
\int_{\partial \Omega} \psi\ {\rm dS} \leq C
\end{equation*}
for any $\psi$, $\psi > 0$ almost everywhere in $\Omega$ satisfying
\begin{equation*}
\| \Grad \log (\psi) \|_{\L^2(\Omega)} \leq \Lambda_1\quad\text{and}\quad
\int_{\partial \Omega} \psi \ {\rm dS} \leq Z \bigl(1 + \| \psi \|^\beta_{\L^s(\Omega)} \bigr)
+ \omega \| \psi \|_{\L^s(\Omega)}.
\end{equation*}
\end{Lemma}

\begin{proof}
Arguing by contradiction, we suppose that there is a sequence $\{ \psi_n \}_{n=1}^\infty$ such that
\begin{subequations}
\begin{align}
&\psi_n > 0 \ \mbox{a.e in}\ \Omega,\\&
\|\log(\psi_n) \|_{\L^2(\Omega)} \leq \Lambda_1,\\&
\label{D8}
\int_{\partial \Omega} \psi_n \ {\rm dS} \leq  Z \bigl(1 + \| \psi_n \|^\beta_{\L^s(\Omega)} \bigr)
+ \omega \| \psi_n \|_{\L^s(\Omega)},\\&
b_n \eqldef \int_{\partial \Omega} \psi_n \ {\rm dS} \to \infty \ \mbox{as}\ n \to \infty.
\end{align}
\end{subequations}
Consider the normalized sequence
\begin{equation*}
\xi_n \eqldef \frac{\psi_n }{\int_{\partial \Omega} \psi_n \ {\rm dS}}.
\end{equation*}
We have
\begin{equation*}
\int_{\partial \Omega} \xi_n\ {\rm{dS}}= 1\quad\text{and}\quad \| \Grad \log (\xi_n) \|_{\L^2(\Omega)} =
\| \Grad \log (\psi_n) \|_{\L^2(\Omega)} \leq \Lambda_1.
\end{equation*}
It follows from Lemma \ref{DL2} that
\begin{equation*}
\| \xi_n \|_{\L^s(\Omega)} \leq \chi(\Lambda_1, 1,s).
\end{equation*}
Dividing \eqref{D8} on $b_n = \int_{\partial \Omega} \psi_n \ {\rm dS}$, we obtain
\begin{equation*}
\begin{aligned}
1 = \int_{\partial \Omega} \xi_n {\rm{dS}} &\leq Z \Bigl( \frac{1}{b_n}
+ \frac{1}{b_n^{1 - \beta}} \| \xi_n \|_{\L^s(\Omega)}^\beta \Bigr)
+ \omega \| \xi_n \|_{\L^s(\Omega)} \\ &\leq
Z \Bigl( \frac{1}{b_n} + \frac{1}{b_n^{1 - \beta}} \chi^{\beta}(\Lambda_1,1,s) \Bigr)
+ \omega \chi(\Lambda_1,1,s) \to \omega \chi(\Lambda_1,1,s) < 1
\ \mbox{as}\ n \to \infty,
\end{aligned}
\end{equation*}
which is a contradiction.
\end{proof}

We apply Lemmas \ref{DL2} and \ref{DL3} to
$\psi = \vtd$, $\Lambda_1$ determined by means of the entropy
estimates \eqref{D2}, and $s$, $\beta$, $Z$ as in \eqref{D7}. In accordance
with Lemma \ref{LD1}, we fix $\Ov{\vu} = \Ov{\vu}_\omega$ so that
\begin{equation*}
c_{D} \| \Ov{\vu}_{\omega} \|_{\L^q(\Omega; \R^2)} < \omega
\end{equation*}
in \eqref{D7}. In accordance with Lemma \ref{DL3}, we conclude that
\begin{equation*}
\intO{ \vtd } \leq c_{D} \ \mbox{uniformly for}\ \delta \to 0,
\end{equation*}
which implies that
\begin{equation}
\label{D9}
\| \log (\vtd) \|_{\W^{1,2}(\Omega)} \leq c_{D}.
\end{equation}

\subsection{Convergence}

At this stage, the same machinery used in Section \ref{TM} allows us to conclude that
\begin{equation}
\label{D10}
\begin{aligned}
\| \vtd^\beta \|_{\L^s(\Omega)} &\leq c(\beta, s) \ \mbox{for any}\ \beta \in \R, \ 1 \leq s < \infty,\\
\| \Grad \vtd \|_{\L^r(\Omega)} + \| \Grad \log(\vtd) \|_{\L^2(\Omega)} &\leq c(r)
\ \mbox{for any}\ 1 \leq r < 2,\\
\| \vud \|_{\W^{1,r}(\Omega; \R^2)} &\leq c(r) \ \mbox{for any}\ 1 \leq r < 2.
\end{aligned}
\end{equation}
The uniform bounds \eqref{D10}, together with
\begin{equation*}
0 \leq \vrd < \Ov{\vr}\ \mbox{a.e in} \ \Omega
\end{equation*}
are strong enough to perform the limit passage in the equations by using the same arguments
as in Section \ref{E}. We have completed the proof of Theorem \ref{TM1}.

\section{Concluding remarks}
\label{C}

We have considered the EOS of the form
\begin{equation*}
p(\vr, \vt) \eqldef
\vr \vt h(\vr)\quad\text{and}\quad e(\vr, \vt) \eqldef c_v \vt.
\end{equation*}
In view of the fact that the density is {\it a priori} bounded and the rather strong estimates
on the temperature, the result may be extended to more general pressure law including finite
"virial series perturbation'' of the form
\begin{equation*}
p(\vr, \vt) \eqldef
\vr \vt h(\vr) + \sum_{m=1}^M b_m(\vt) \vr^{\alpha_m},\quad \alpha_m \geq 0, \quad
0 \leq b_m(\vt) \aleq \vt^{-\beta} + \vt^\beta,\quad \beta \geq 0.
\end{equation*}
Monotonicity of the pressure with respect to the density plays a crucial for stationary problems therefore the method cannot be adapted to pressure laws that are
non--monotone with respect to the density.

The asymptotic behavior of the transport coefficients could be possibly relaxed to
\begin{equation*}
\mu(\vt) \approx (1 + \vt^{\alpha_1}), \quad
\lambda(\vt) \approx (1 + \vt^{\alpha_2}),\quad \kappa (1 + \vt^{\alpha_3}),\quad 0 \leq \alpha_i < 1.
\end{equation*}
In view of the estimates in Section \ref{E}, however, the sublinear growth seems essential.

The proof depends heavily on the estimates \eqref{I11} pertinent to planar domains.
Extension to the 3-D case would be definitely limited by the available {\it a priori}
bounds on the temperature and possibly require stronger hypothesis imposed on
both the EOS
and the transport coefficients.

\renewcommand{\arraystretch}{0.91}{\small{
\paragraph*{Acknowledgments}
The research of E.F.~leading to these results has received funding from the
Czech Sciences Foundation (GA\v CR), Grant Agreement
18--05974S. The Institute of Mathematics of the Academy of Sciences of
the Czech Republic is supported by RVO:67985840.
This research was performed during the stay of E.F. as an invited professor at the INSA--Lyon.}}

\end{document}